\documentclass[a4paper,11pt,leqno]{article}

\input{xy}
\xyoption{all}
\xyoption{poly}
\usepackage[all]{xy}
\usepackage{amsmath,amsthm,amsfonts,amssymb,stmaryrd,fancyhdr,mathtools}
\usepackage{latexsym}
\usepackage{color}
\usepackage{graphicx}
\usepackage{float}  
\usepackage[utf8]{inputenc}
\usepackage{fullpage}
\usepackage[small, bf, margin=90pt, tableposition=bottom]{caption}
\usepackage[T1]{fontenc}
\usepackage{enumerate}
\usepackage{verbatim}
\usepackage{mathrsfs}
  \usepackage{hyperref}
\hypersetup{
    colorlinks=true, 
    linkcolor=blue, 
    citecolor=green,
    urlcolor=blue, 
}

\theoremstyle{definition} 
\newtheorem{theorem}{Theorem}[section]
\newtheorem{theorem*}{Theorem}
\newtheorem{corollary}[theorem]{Corollary}
\newtheorem{lemma}[theorem]{Lemma}

\theoremstyle{definition}        
\newtheorem{proposition}[theorem]{Proposition}
\theoremstyle{definition}        
\newtheorem{definition}[theorem]{Definition}
\theoremstyle{definition}        
\newtheorem{remark}[theorem]{Remark}

\numberwithin{equation}{section}                    

\usepackage{fancyhdr}
\fancypagestyle{plain}{%
\fancyhead{} 
 }

\newcommand\RR{{\mathbb{R}}}

\newcommand\NN{{\mathbb{N}}}

\newcommand\supx{{\sup_{x\in\RR^3}}}
\newcommand\dxi{{\partial_{x_i}}}
\newcommand\dxj{{\partial_{x_j}}}
\newcommand\xij{{x_i x_j}}
\newcommand\xxi{{{x_i}}}
\newcommand\xxj{{{x_j}}}
\newcommand\dt{{\frac{d}{dt}}}

\newcommand\intt{{\int_0^t}}
\newcommand\intr{{\int_{\RR^3}}}
\newcommand\kk{{ ^{(k)} }}
\newcommand\ek{{^{(k-1)}}}
\newcommand\ie{{{\textsl{z}}^{-1}}}
\newcommand\z{{\textsl{z}}}
\newcommand\C{{\eta}}

\def\P{{\mathbb{P}}}

\def\N{\ensuremath{\mathbb N}}

\def\P{\ensuremath{\mathbb P}}

\def\F{{\mathcal F}}

\def\W{{\mathcal W}}

\makeatletter
\newcommand{\xdashrightarrow}[2][]{\ext@arrow 0359\rightarrowfill@@{#1}{#2}}
\newcommand{\xdashleftarrow}[2][]{\ext@arrow 3095\leftarrowfill@@{#1}{#2}}
\newcommand{\xdashleftrightarrow}[2][]{\ext@arrow 3359\leftrightarrowfill@@{#1}{#2}}
\def\rightarrowfill@@{\arrowfill@@\relax\relbar\rightarrow}
\def\leftarrowfill@@{\arrowfill@@\leftarrow\relbar\relax}
\def\leftrightarrowfill@@{\arrowfill@@\leftarrow\relbar\rightarrow}
\def\arrowfill@@#1#2#3#4{%
  $\m@th\thickmuskip0mu\medmuskip\thickmuskip\thinmuskip\thickmuskip
   \relax#4#1
   \xleaders\hbox{$#4#2$}\hfill
   #3$%
}
\makeatother

\title{On  local well-posedness of the stochastic  incompressible density-dependent Euler equations}

\begin{document}

\title{On local well-posedness of the stochastic  incompressible density-dependent Euler equations}
\author{Christian Olivera\footnote{Departamento de Matemática, IMECC, Universidade Estadual de Campinas, SP-Brazil.\\
\href{mailto:colivera@ime.unicamp.br}{colivera@ime.unicamp.br}} , Claudia Espitia\footnote{Departamento de Matemática, IMECC, Universidade Estadual de Campinas,
SP-Brazil.\\\href{mailto:cespitia@ime.unicamp.br}{cespitia@ime.unicamp.br}} , David A. C. Mollinedo\footnote{Departamento Acadêmico de Matemática, Universidade Tecnológica Federal do Paraná, PR-Brazil.\\ \href{mailto:davida@utfpr.edu.br}{davida@utfpr.edu.br}}}

\date{}

\maketitle
\begin{abstract}
In this paper we study the stochastic inhomogeneous incompressible Euler equations in the whole space  $\RR^3$. We prove the existence and pathwise uniqueness of local solutions with both additive and multiplicative stochastic noise.  Our approach is based on reducing our problem to a random problem and some estimations for  type transport equations.

\end{abstract}

 \tableofcontents

\pagebreak
 
 \section{Introduction}

 In this work, we are concerned with the following models of incompressible inviscid fluid with variable density:

\begin{equation}\label{eq0}
\begin{cases}
d\rho +  v \cdot \nabla \rho \,dt=0,&\\
 d v + (v\cdot \nabla)v \,dt+ \frac{\nabla \pi}{\rho}\,dt = -v\circ\,d \W, & \\
\text{div } v =0, &
\end{cases}
\end{equation}

\begin{equation}\label{eq0bi}
\begin{cases}
d\rho +  v \cdot \nabla \rho \,dt=0,&\\
 d v + (v\cdot \nabla)v \,dt+ \frac{\nabla \pi}{\rho}\,dt = d \W^{Q}, & \\
\text{div } v =0, &
\end{cases}
\end{equation}
with initial conditions given by
\begin{align*}
    \left\{
    \begin{aligned}
    &\rho\big|_{t=0}=\rho_0(x),\\
        &v\big|_{t=0} = v_0(x),
    \end{aligned}
    \right.
\end{align*}
in  $\Omega \times [0,T] \times \RR^3$, for some $T>0$, where the velocity vector field $v$, and the scalar fields of density $\rho$ and pressure 
$\pi$ are the unknowns. Here, we work within a stochastic framework, considering a stochastic basis $(\Omega, \F, (\F_t)_{t\geq 0},  \P)$ with a complete right-continuous filtration,  $\W$ a standard scalar valued Brownian Motion
 and  $\W^{Q}$ is a $Q$-Cylindrical Brownian Motion over  $W^{k,2}(\RR^3)$, with $k>5/2$.  In equation \eqref{eq0}, $\W$ denotes the standard scalar valued Wiener process while the symbol $\circ$ indicates that the stochastic integral in the weak formulation of the problem is interpreted in the Stratonovich sense.\\

Among all the equations used in fluid dynamics, the Euler equations are the classical model for the motion of an inviscid, incompressible fluid. 
The addition of stochastic terms to the governing equations is frequently used to capture numerical, empirical, and physical uncertainties in fields spanning from climatology to turbulence theory.
We also recall that these stochastic Euler equations can be viewed as the limit of stochastic particle systems with moderate iteration, see \cite{Correa} and 
\cite{Correa2}. The Stratonovich form is the natural one for several reasons, including physical intuition related to the Wong–Zakai principle and is the right one to deal with manifold-valued SPDEs, see \cite{Bre}. \\

The deterministic incompressible density-dependent Euler equations have been extensively investigated.
We will recall only the results on the local existence and uniqueness of solutions.  Local   well-posedness results, in bounded domains,  were obtained in  \cite{Veiga, Veiga2} by considering initial data of class  $C^{\infty}$.  Later,  \cite{ITO}  proved the well-posedness in the Sobolev space $W^{2,p}(U)$  with $U$ being either a bounded or unbounded smooth domain of $\RR^{3}$. In the works \cite{ITO2} and  \cite{ITO3}, the author  considered the whole space  $\RR^{3}$  and obtained local existence of solution with initial condition in  $H^{s}$. For results in Besov Space, see \cite{Chae}, \cite{Danchin},  \cite{Ferreira} and  \cite{Has}.\\

The literature on the stochastic homogeneous incompressible Euler equations consists of a number of works, including \cite{Bessa}, \cite{Bessa2}, \cite{Bre3}, \cite{Bre2},  \cite{Vicol}, \cite{KIM}, \cite{KIM2}, \cite{Mi}. They show, among other things, global well-posedness for the $2D$ Euler equations with additive and multiplicative noise and local well-posedness for $d=3$.
However fundamental questions of well–posedness and even more existence of solutions to problems dealing with stochastic perturbations of  inhomogeneous incompressible fluids are, to the best of our knowledge, open. \\

Our goal is to establish the local existence and uniqueness   of $W^{2,p}(\RR^{3})-$solution of the stochastic
Euler equations (\ref{eq0}) and (\ref{eq0bi}).  As in the deterministic case, global existence and uniqueness is a famous open problem. This naturally leads to the question of whether noise can restore well-posedness, a concept often referred to as regularization by noise. There is evidence that stochastic perturbations, particularly those involving transport, may have a regularizing influence, see e.g. \cite{Bagnara-Maurelli-Xu, Fedrizzi, Flandoli2}. However, such effects are not universal, as shown in \cite{Breit2, Chi}, where solutions still develop singularities in finite time. \\

More precisely, the aim of the present paper is to prove the following theorems: the first one concerns the case of multiplicative noise, while the second one deals with additive noise.

\begin{theorem}\label{maintheorem}
    Suppose that for $3<p<\infty$,
 the function $\rho_0$ satisfies
 \begin{align}\label{initialconditionsrho}
  \rho_0 \in C(\RR^3),\quad 0< m \le \rho_0 \leq M, \quad  \nabla\rho_0 \in W^{1,p}(\RR^3),
 \end{align}
 and the vector field $v_0$ verifies
 \begin{align}\label{initialconditions-v0}
   v_0 \in W^{2,p}(\RR^3)\,, \quad \text{div } v_0 =0\,.
 \end{align}
Then, there exists a unique local pathwise solution $(\rho,\nabla \pi, v, \tau)$ for Problem \eqref{eq0} in the sense of Definitions \ref{deflocalsolution} and \ref{localuniqueness}. Moreover, there also exists a unique maximal pathwise solution in the sense of Definitions \ref{defmaximalsolution} and \ref{maxuniqueness}.

 \end{theorem}

\begin{theorem}\label{maintheorembi} 
We assume 
that $\W^Q=(\W_{1}^Q, \W_{2}^Q, \W_{3}^Q)$  where 
$\W_{i}^Q$ are  $Q$-Cylindrical Brownian Motion  in 
$W^{k,2}(\RR^3)$, where $Q$ is trace class,  $k> 5/2$ and $\text{div} \  \W^Q =0$.  Moreover,  we suppose that for $3<p<\infty$,
 the function $\rho_0$ satisfies
 \begin{align}\label{initialconditionsrhobi}
  \rho_0 \in C(\RR^3),\quad 0< m \le \rho_0 \leq M, \quad  \nabla\rho_0 \in W^{1,p}(\RR^3),
 \end{align}
 and the vector field $v_0$ verifies
 \begin{align}\label{initialconditions-v0vi}
   v_0 \in W^{2,p}(\RR^3)\,, \quad \text{div } v_0 =0\,.
 \end{align}
Then, there exists a unique local pathwise solution $(\rho,\nabla \pi, v, \tau)$ for Problem \eqref{eq0bi} in the sense of Definition \ref{deflocalsolutionbi} and \ref{localuniqueness}. Moreover, there also exists a unique maximal pathwise solution in the sense of Definitions \ref{defmaximalsolutionbi} and \ref{maxuniqueness}.

 \end{theorem}

The proofs are based on estimations for the type transport equation in $W^{2,p}$ following the framework outlined in \cite{ITO}.  We also completed several proofs that were not fully addressed in \cite{ITO}.\\

The paper is organized as follows: In section 2 we consider the space where we are going to work and define the notion of solution to the stochastic Euler equations \eqref{eq0} and \eqref{eq0bi}. In Section 3, we study auxiliary problems associated with these equations and derive the corresponding a priori estimates. These estimates are crucial for analyzing the sequence of successive approximations defined in the subsequent sections. In Section 4, we prove our first main result, Theorem \ref{maintheorem}, concerning the case of multiplicative noise. We begin by constructing a sequence of successive approximations related to the auxiliary problems introduced earlier. Then, using the a priori estimates, we establish the convergence of this sequence and, as a consequence, obtain the existence and uniqueness results for problem \eqref{eq0}. We conclude Section 4 by proving the existence of a maximal pathwise solution. Finally, in Section 5, we address the case of additive noise. 


In closing, we would like to mention that future research will address to study inhomogeneous Euler systems with more general noise.


\section{Space and definition of solutions}

In this section, we introduce Sobolev spaces and the concepts of both local and maximal pathwise solutions for our problems \eqref{eq0} and \eqref{eq0bi}.

For $k \in  \NN \cup \left\{  0 \right\} $ and   $1 \leq p < \infty$, the space $W^{k,p}(\RR^d)$ consists of functions $u\in L^{p}(\RR^{d})$ such that the weak derivatives $D^{m}u$ exists and are in $L^{p}(\RR^{d})$ for all multi-indices $m$ such that
$|m|\leq k$. On this space we define the Sobolev norm

\[
\| u\|_{k,p}= \sum_{|m|=0}^{k}  \|D^{m} u\|_{L^{p}(\RR^{d})}.
\]


\begin{definition}[Local pathwise solution]\label{deflocalsolution}
 Let  $(\Omega, \F, (\F_t)_{t\geq 0},  \P)$ be a stochastic basis with a complete right-continuous filtration, and $\W$ 
 a real Brownian motion relative to the filtration. 
 \begin{itemize}
     \item  Assuming the conditions $3<p<\infty$ and $(\rho_0, v_0)\in C(\RR^3)\times W^{2,p}(\RR^3)$, we call $(\rho, \nabla \pi, v, \tau)$ a local pathwise solution to the system \eqref{eq0} provided
 \begin{itemize}
     \item [(i)]$\tau$ is an a.s. strictly positive $(\F_t)$-stopping time;
     \item [(ii)] the density $\rho$ is a $C(\RR^3)$- progressively measurable process satisfying
     \begin{align*} 
   0< m \le \rho(\cdot \wedge \tau) \leq M, \quad  \nabla \rho(\cdot\wedge \tau)\in C\left([0, T];W^{1,p}(\RR^3)\right)\hspace{0.5cm}\P-a.s
     \end{align*}
     for some constants $m$ and $M$;
     \item[(iii)] the pressure $\nabla \pi$ is a $W^{2,p}(\RR^3)$- progressively measurable process satisfying
     \begin{equation*}
          \nabla \pi(\cdot\wedge \tau)\in C\left([0, T];W^{2,p}(\RR^3)\right)\hspace{0.5cm}\P-a.s;
     \end{equation*}
     \item [(iv)] the velocity $v$ is a $W^{2,p}(\RR^3)$- progressively measurable process satisfying
     \begin{equation*}
         v(\cdot\wedge \tau)\in C\left([0, T];W^{2,p}(\RR^3)\right)\hspace{0.5cm}\P-a.s;
     \end{equation*}
     \item [(v)] there holds $\P$-a.s. 
    \begin{gather*}
        \rho(t\wedge \tau)=\rho_0-\int_0^{t\wedge \tau} v\cdot \nabla \rho\,ds,\\
        v(t\wedge \tau)=v_0-\int_0^{t\wedge \tau}v\cdot \nabla v\,ds-\int_0^{t\wedge \tau}\frac{\nabla \pi}{\rho}\,ds-\int_0^{t\wedge \tau} v\circ\,d\W_s,\\
        \text{div }v=0,
    \end{gather*}
      for every $t\in[0,T]$.
    
 \end{itemize}
 
 \end{itemize}

\end{definition}

\begin{definition}[Local pathwise solution]\label{deflocalsolutionbi}
 Let  $(\Omega, \F, (\F_t)_{t\geq 0},  \P)$ be a stochastic basis with a complete right-continuous filtration, and $\W^{Q}$  is a $Q$-Cylindrical Brownian Motion over  $W^{k,2}(\RR^3)$ relative to the filtration. 
 \begin{itemize}
     \item  Assuming the conditions $3<p<\infty$ and $(\rho_0, v_0)\in C(\RR^3)\times W^{2,p}(\RR^3)$, we call $(\rho, \nabla \pi, v, \tau)$ a local pathwise solution to the system \eqref{eq0bi} provided
 \begin{itemize}
     \item [(i)]$\tau$ is an a.s. strictly positive $(\F_t)$-stopping time;
     \item [(ii)] the density $\rho$ is a $C(\RR^3)$- progressively measurable process satisfying
     \begin{align*} 
   0< m \le \rho(\cdot \wedge \tau) \leq M, \quad  \nabla \rho(\cdot\wedge \tau)\in C\left([0, T];W^{1,p}(\RR^3)\right)\hspace{0.5cm}\P-a.s
     \end{align*}
     for some constants $m$ and $M$;
     \item[(iii)] the pressure $\nabla \pi$ is a $W^{2,p}(\RR^3)$- progressively measurable process satisfying
     \begin{equation*}
          \nabla \pi(\cdot\wedge \tau)\in C\left([0, T];W^{2,p}(\RR^3)\right)\hspace{0.5cm}\P-a.s;
     \end{equation*}
     \item [(iv)] the velocity $v$ is a $W^{2,p}(\RR^3)$- progressively measurable process satisfying
     \begin{equation*}
         v(\cdot\wedge \tau)\in C\left([0, T];W^{2,p}(\RR^3)\right)\hspace{0.5cm}\P-a.s;
     \end{equation*}
     \item [(v)] there holds $\P$-a.s. 
    \begin{gather*}
        \rho(t\wedge \tau)=\rho_0-\int_0^{t\wedge \tau} v\cdot \nabla \rho\,ds,\\
        v(t\wedge \tau)=v_0-\int_0^{t\wedge \tau}v\cdot \nabla v\,ds-\int_0^{t\wedge \tau}\frac{\nabla \pi}{\rho}\,ds  + \W_{t\wedge \tau}^{Q},\\
        \text{div }v=0,
    \end{gather*}
      for every $t\in[0,T]$.
 \end{itemize}
 \end{itemize}
\end{definition}

\begin{definition}[Uniqueness]\label{localuniqueness}

    We say that local pathwise solutions are unique if, given any pair 
   \begin{equation*}
        (\rho^1, \nabla \pi^1, v^1, \tau^1), \quad (\rho^2, \nabla \pi^2, v^2, \tau^2)
   \end{equation*}
   of local pathwise solutions with the same initial data $(\rho_0, v_0)$, it holds that
   \begin{equation*}
     \mathbb{P}\Big( \, (\rho^1,\nabla \pi^1, v^1)(t)
      =(\rho^2,\nabla \pi^2, v^2)(t), \ \forall\, t\in [0, \tau^1\wedge\tau^2] \,\Big)=1.  
   \end{equation*}
\end{definition}

\begin{remark}
It is worth noting that the solutions introduced above are weak in the PDE sense—partial derivatives are interpreted in the sense of distributions—but strong
in the stochastic sense—the stochastic integral is considered on the original probability space.  However,  as a direct consequence of the  theorems \ref{maintheorem} and\ref{maintheorembi}  and  classical embedding theorems the solutions obtained here are strong in the PDE sense.  Moreover, the Stratonovich integral in equation \eqref{eq0}  is understood in a finite-dimensional sense: for each fixed spatial point $x \in \mathbb{R}^3$, the integral reduces to the classical finite-dimensional Stratonovich integral.  
\end{remark}
\begin{remark}
Even more, the embedding theorems imply that the norms $\|\W^{Q} \|_{C^2}$ and $\|\W^{Q} \|_{2,p}$ are bounded by the norm $\|\W^{Q}\|_{k,2}$, with $k>5/2$. This  hypothesis  will be used in the proof of Proposition~\ref{propubi}.
\end{remark}

\begin{definition}[Maximal pathwise solution]\label{defmaximalsolution}
    Fix a stochastic basis with a real Brownian Motion and initial conditions as in Definition \ref{deflocalsolution}. We call $(\rho, \nabla \pi, v, \xi)$ a maximal pathwise solution to the system \eqref{eq0} if the following hold:
\begin{itemize}
    \item[(i)] There exists a sequence of $(\F_t)$-stopping times $\left(\sigma_n\right)_{n\in \mathbb{N}}$, monotone increasing and converging to $\xi$ $\P$-a.s., such that each quadruple $(\rho, \nabla \pi, v, \sigma_n)$, $n \in \mathbb{N}$, is a local pathwise solution.
    \item[(ii)] If $(\rho', \nabla \pi', v', \xi')$ is another solution satisfying the above property, then $\xi \leq \xi'$ $\P$-a.s. implies $\xi = \xi'$ $\P$-a.s. 
\end{itemize}
\end{definition}
\begin{definition}[Maximal pathwise solution]\label{defmaximalsolutionbi}
    Fix a stochastic basis with  $\W^{Q}$ a $Q$-Cylindrical Brownian Motion over $W^{k,2}(\RR^3)$ relative to the filtration,  and initial conditions as in Definition \ref{deflocalsolutionbi}. We call $(\rho, \nabla \pi, v, \xi)$ a maximal pathwise solution to the system \eqref{eq0bi} if the following hold:
\begin{itemize}
    \item[(i)] There exists a sequence of $(\F_t)$-stopping times $\left(\sigma_n\right)_{n\in \mathbb{N}}$, monotone increasing and converging to $\xi$ $\P$-a.s., such that each quadruple $(\rho, \nabla \pi, v, \sigma_n)$, $n \in \mathbb{N}$, is a local pathwise solution.
    \item[(ii)] If $(\rho', \nabla \pi', v', \xi')$ is another solution satisfying the above property, then $\xi \leq \xi'$ $\P$-a.s. implies $\xi = \xi'$ $\P$-a.s. 
\end{itemize}
\end{definition}
\begin{definition}[Uniqueness]\label{maxuniqueness}
We say that maximal pathwise solutions are unique if, given any pair 
   \begin{equation*}
     (\rho^1, \nabla \pi^1, v^1, \xi^1), \quad (\rho^2, \nabla \pi^2, v^2, \xi^2)  
   \end{equation*}
   of maximal pathwise solutions with the same initial data $(\rho_0, v_0)$, we have $\xi^1=\xi^2$ almost surely and
   \begin{equation*}
     \mathbb{P}\Big( \, (\rho^1,\nabla \pi^1, v^1)(t)
      =(\rho^2,\nabla \pi^2, v^2)(t), \ \forall\, t\in [0, \xi^1) \,\Big)=1.  
   \end{equation*}  
\end{definition}

\medskip

\section{A priori estimates }



In this section, we derive several a priori estimates that will play a central role in the analysis of the problems. To streamline calculations, we will employ a change of variables in the second equation of problem \eqref{eq0}, enabling us to handle a random equation instead of a stochastic one.
Let us define a new variable $\Tilde{v}$ as
\begin{equation}\label{cvariables}
    \Tilde{v}=\text{exp}(\W(t))v.
\end{equation}
Using the well-known chain rule for  the Stratonovith integral, it is easy to prove that the stochastic Euler equation
 $$d v + (v\cdot \nabla)v \,dt+ \frac{\nabla p}{\rho}\,dt = -v\circ\,d \W_t$$
transforms into this new random equation
\begin{equation}\label{eqvtilde}
    d\Tilde{v}=-\text{exp}\left(-\W(t)\right) (\Tilde{v}\cdot \nabla) \Tilde{v}\,dt-\frac{1}{\rho}\text{exp}\left( \W(t)\right) \nabla p\,dt.
\end{equation}

\subsection{ A priori estimates for transport equation.}

Concerning the transport equation, we first establish its well-posedness and then derive several important estimates relevant for the subsequent analysis.

\begin{proposition}\label{proprho}
    If $v\in   C([0,T];W^{2,p}(\RR^3))$ satisfies div $v=0$ and 
\begin{equation*}
    \|v\|_{L_\infty([0,T]\times\RR^3)}+\int_0^T \|\nabla v(t)\|_{L_\infty(\RR^3)}\,dt<\infty,
\end{equation*}
then for any $\rho_{0}\in C^1(\RR^3)$ such that $0<m\leq \rho_{0}\leq M< \infty$, the problem 
\begin{equation}\label{eq1}
\begin{cases} 
\rho_t + v \cdot \nabla \rho =0, &\\ 
\rho|_{t=0}=\rho_0(x),&
\end{cases}
\end{equation}
has a unique solution $\rho\in C^{1,1}([0,T]\times\RR^3)$ which satisfies
\begin{equation}\label{estrho1}
    m\leq \rho\leq M.
\end{equation}
\end{proposition}

\begin{proof}
By applying the classical method of characteristics we have that there exists a unique solution $\rho\in C^{1,1}([0,T]\times\RR^3)$ of the Eq. \eqref{eq1} given by
\begin{equation}\label{Sol-TE}
\rho(t,x)=\rho_0(\phi^{-1}(t,x)),
\end{equation}
where $\phi(t,x)$ is the flow associated to the vector field $v(t,x)$ through the ordinary differential equation
\begin{align*}
\left\{
\begin{aligned}
& \frac{d\phi}{dt}  =v(t,\phi(t,x))\\
& \phi(0,x)  =x\,.
\end{aligned}
\right.
\end{align*}
Thus, by the representation \eqref{Sol-TE}, it is clear that  the estimate \eqref{estrho1} holds. For more details about the method of characteristics, for example, to see  \cite{LS}.

\end{proof}

The following result is a suitable adaptation of the estimates described in \cite[Lemmas 2.2]{ITO}, without considering the corresponding boundary conditions.

\begin{proposition}\label{proprho2}
  The above solution satisfies

\begin{equation}\label{estrho2}
     \|\nabla \rho\|_{L_\infty([0,T]\times\RR^3)}\leq \sqrt{3} \|\nabla \rho_0\|_{L_\infty(\RR^3)}\text{ exp}\left(\int_0^T \|\nabla v(t)\|_{L_\infty(\RR^3)}\,dt\right),
\end{equation}
\begin{equation}\label{estrho3}
      \|\rho_t\|_{L_\infty([0,T]\times\RR^3)}\leq \sqrt{3}  \|v\|_{L_\infty([0,T]\times\RR^3)}\|\nabla \rho_0\|_{L_\infty(\RR^3)}\text{ exp}\left(\int_0^T \|\nabla v(t)\|_{L_\infty( \RR^3)}\,dt\right).
\end{equation}
Moreover, if $\nabla\rho_0\in W^{1,p}(\RR^3)$ and $v \in  L^1(0, T; W^{2,p}(\RR^3))$, then
\begin{equation}\label{esrho}
    \|\nabla\rho(t)\|_{1,p}\leq  \|\nabla\rho_0\|_{1,p}\text{ exp}\left(c_1 \intt\|v(s)\|_{2,p} \,ds\right).
\end{equation}
\end{proposition}

\begin{proof}

The first two estimates are classical and can be consulted in \cite{LS}. For the last one, notice that
    \begin{equation*}
        \dt\left(\frac{\partial}{\partial_{x_i}} \rho\right)=\frac{\partial}{\partial_{x_i}}\left(\dt\rho\right)=-\frac{\partial}{\partial_{x_i}}\left( v \cdot \nabla \rho\right)=-\left(v_{x_i} \cdot \nabla \rho+v \cdot \nabla \rho_{x_i}\right),
    \end{equation*}
  and   then 
\begin{align*}
  \frac{1}{p}\dt|\dxi\rho|^p&=|\dxi\rho|^{p-1}sgn(\dxi\rho)\dt\left
 (\frac{\partial}{\dxi}\rho\right)\\ 
 &=-|\dxi\rho|^{p-1}sgn(\dxi\rho)\left(v_{x_i} \cdot \nabla \rho+v \cdot \nabla \rho_{x_i}\right).
\end{align*}
After integrating on $\RR^3$, the corresponding second term vanishes using that $\text{div }v=0$,  and we can estimate the first derivative as follows using the Sobolev embedding
\begin{align*}
    \frac{1}{p}\dt\|\dxi\rho(t)\|^p_{L^p}&=-\int_{\RR^3}|\dxi\rho|^{p-1}sgn(\dxi\rho) v_{x_i}\cdot \nabla\rho\,dx\\
    &\hspace{0.42cm}-\int_{\RR^3}|\dxi\rho|^{p-1}sgn(\dxi\rho) v\cdot \nabla\rho_{x_i}\,dx\\
    &\leq  \supx|v_{x_i}| \int_{\RR^3}|\dxi\rho|^{p-1}|\nabla \rho|\,dx\\
    &\leq  \|v(t)\|_{2,p}  \|\nabla \rho(t) \|^p_{1,p}.
\end{align*}
Now, the second derivative fulfills 
\begin{align*}
  \dt \left(\frac{\partial^2}{\dxi\dxj}\rho\right)&= \frac{\partial^2}{\dxi\dxj}\left(\dt\rho\right)=-\frac{\partial}{\dxj}\left(v_{x_i} \cdot \nabla \rho+v \cdot \nabla \rho_{x_i}\right)\\
  &=-\left(v_\xij \cdot \nabla \rho+v_\xxi
  \cdot \nabla \rho_\xxj+v_\xxj\cdot\nabla\rho_\xxi+v\cdot\nabla\rho_\xij\right),
\end{align*}
then
\begin{align*}
\frac{1}{p}\dt|\partial_\xij\rho|^p&=|\partial_\xij\rho|^{p-1}sgn(\partial_\xij\rho)\dt\left(\frac{\partial^2}{\dxi\dxj}\rho\right)\\  
 &=-  |\partial_\xij\rho|^{p-1}sgn(\partial_\xij\rho)\left(v_\xij \cdot \nabla \rho+v_\xxi
  \cdot \nabla \rho_\xxj+v_\xxj\cdot\nabla\rho_\xxi+v\cdot\nabla\rho_\xij\right).
\end{align*}
Again, integrating on $\RR^3$ and using Hölder's inequality  we have
\begin{align*}
    \frac{1}{p}\dt\|\partial_\xij\rho(t)\|^p&\leq  \int_{\RR^3} |\partial_\xij\rho|^{p-1}|v_\xij|  |\nabla \rho|\,dx\\
    &+ \int_{\RR^3} |\partial_\xij\rho|^{p-1}|v_\xxi|
  | \nabla \rho_\xxj|\,dx\\
  &+ \int_{\RR^3} |\partial_\xij\rho|^{p-1}|v_\xxj||\nabla\rho_\xxi| \,dx\\
  &\leq   \supx |\nabla \rho|\left( \int_{\RR^3}|\partial_\xij\rho|^{(p-1)\frac{p}{p-1}}\,dx \right)^{\frac{p-1}{p}}\left( \int_{\RR^3}|v_\xij|^p\,dx \right)^{\frac{1}{p}}\\
  &+  \supx |v_\xxi|\left( \int_{\RR^3}|\partial_\xij\rho|^{(p-1)\frac{p}{p-1}}\,dx \right)^{\frac{p-1}{p}}\left( \int_{\RR^3}|\nabla \rho_\xxj|^p\,dx \right)^{\frac{1}{p}}\\
  &+    \supx |v_\xxj|\left( \int_{\RR^3}|\partial_\xij\rho|^{(p-1)\frac{p}{p-1}}\,dx \right)^{\frac{p-1}{p}}\left( \int_{\RR^3}|\nabla \rho_\xxi|^p\,dx \right)^{\frac{1}{p}}.\\
\end{align*}
As a consequence of the the Sobolev embedding  we get 
\begin{align*}
 \frac{1}{p}\dt\|\partial_\xij\rho(t)\|^p\leq  3\|\nabla\rho(t)\|^p_{1,p} \|v(t)\|_{2,p}. 
\end{align*}
Finally,  Gronwall's inequality implies our claim.

\end{proof}

\subsection{ A priori estimates for elliptic equations.}

Since it is assumed $\text{div }v=\text{div }\Tilde{v}=0$,  if we denote $\z(t)=\text{exp}\left( \W(t)\right)$ and apply the divergence operator in the  equations  (\ref{eqvtilde}) and (\ref{eq0bi}), then we get the random elliptic equations described in the following propositions.

\begin{proposition}\label{propp}
    Let $\rho\in C^{1,1}([0,T]\times\RR^3)$ such that $m\leq \rho\leq M<\infty$ and $\nabla\rho\in C([0,T];W^{1,p}(\RR^3))$. Also assume that $\Tilde{v}\in  C([0,T];W^{2,p}(\RR^3))$ and $\text{div } v=0$. Then, the equation
\begin{equation*}\label{eq2bi}
\text{div}( \rho^{-1}   \nabla \pi)=-\z^{-2}\sum_{i, j=1}^3 \, \Tilde{v}^i_{x_j}\Tilde{v}^j_{x_i},
\end{equation*}
has a unique solution $\nabla \pi\in C([0,T];W^{2,p}(\RR^3))$, satisfying
\begin{equation*}
    \|\nabla \pi(t)\|_{2,p}\leq \z^{-2}(t )K(\|\nabla\rho(t)\|_{1,p})  \|\Tilde{v}(t)\|^2_{2,p} .
\end{equation*}
where $K$ is a non-decreasing function of $\|\nabla\rho(t)\|_{1,p}$, depending on $m$ and $M$.
\end{proposition}

\begin{proof}
The proof directly follows from the application of \cite[Lemmas 3.1]{ITO} considering $v=\ie \Tilde{v}$.
\end{proof}

\begin{proposition}\label{proppbi}
    Let $\rho\in C^{1,1}([0,T]\times\RR^3)$ such that $m\leq \rho\leq M<\infty$ and $\nabla\rho\in C([0,T];W^{1,p}(\RR^3))$. Also assume that $\Tilde{v}\in  C([0,T];W^{2,p}(\RR^3))$ and $\text{div } v=0$. Then, the equation
\begin{equation*}\label{eq2}
\text{div}( \rho^{-1}   \nabla \pi)=-\sum_{i, j=1}^3 \, v^i_{x_j} v^j_{x_i},
\end{equation*}
has a unique solution $\nabla \pi\in C([0,T];W^{2,p}(\RR^3))$, satisfying
\begin{equation*}
    \|\nabla \pi(t)\|_{2,p}\leq K(\|\nabla\rho(t)\|_{1,p})  \|v(t)\|^2_{2,p}.
\end{equation*}
where $K$ is a non-decreasing function of $\|\nabla\rho(t)\|_{1,p}$, depending on $m$ and $M$.
\end{proposition}
\begin{proof}
The proof directly follows from the application of \cite[Lemmas 3.1]{ITO}. 
\end{proof}

\begin{remark}
Notice that, due to the specific structure of the noise, we can construct solutions in which the pressure is a scalar field of finite variation. Applying the divergence operator to equation \eqref{eq0} and \eqref{eq0bi} and using the incompressibility condition 
div $v=0$ and div $\W^Q=0$  ensures that the pressure has no martingale component.

\end{remark}

\subsection{ A priori estimates for linear Euler equation.}

Returning to the analysis of the Euler equations, we will consider a slight modification of equations  \eqref{eqvtilde} and \eqref{eq0bi}  that enables us to deal with a random linear equation in a new variable $u$, where our original variables $\rho$, $\pi$, $\Tilde{v}$, and $v$ are fixed.

\begin{proposition}\label{propu}
   Let $\rho$ and $\nabla \pi$ be the unique solutions described in Propositions \ref{proprho} and \ref{propp}, respectively. Also, let $\Tilde{v} \in C([0,T];W^{2,p}(\RR^3))$ with $\text{div }\Tilde{v}=0$. Then, the problem
\begin{equation}\label{eq3}
\begin{cases}
 u_t + \ie(\Tilde{v}\cdot \nabla) u+ \z\frac{\nabla \pi}{\rho}=0 \\
u|_{t=0} =v_0(x), &
\end{cases}
\end{equation}
has a unique solution $u\in C([0,T];W^{2,p}(\RR^3))$. Moreover, $u$ satisfies
\begin{equation*}
    \|u(t)\|_{2,p}\leq \text{exp}\left( c_2\intt \ie(s)\|\Tilde{v}(s)\|_{2,p}\,ds\right)\left[\|v_0\|_{2,p}+ \intt \ie(s) L(\|\nabla \rho(s)\|_{1,p})\| \Tilde{v}(s)\|^2_{2,p}\,ds\right].
\end{equation*}

\end{proposition}
\begin{proof}
The well-posedness of this problem comes from a classical result. 
We proceed to prove the corresponding estimate. We do this in several steps and following the same ideas of the Proposition \ref{proprho2}:
\\ 
\noindent{\it \underline{Step 1}: Transport equation.} If we consider $u=(u^1, u^2, u^3)$ we can rewritten the Eq. \eqref{eq3} componentwise by:
\begin{align*}
\left\{
\begin{aligned}
& u^{i}_t + \ie\Tilde{v}\cdot \nabla u^i= -\z \ \frac{ \pi_{x_i}}{\rho} \\
& u^i|_{t=0} =v^i_0(x), 
\end{aligned}
\,,\quad \text{ with }\quad i=1,2,3.
\right.
\end{align*}
\\ 
\noindent{\it \underline{Step 2}: Estimate of $\displaystyle \frac{1}{p}\frac{d}{dt}\|u(t)\|_{L^p}^p$.} Here we also proceed componentwise. Thus, for $i=1$, we have,
\begin{align*}
    \frac{1}{p}\dt|u^1|^p&
    =-\ie|u^1|^{p-1} sgn(u^1) \Tilde{v}\cdot\nabla u^1
    -\z \rho^{-1}|u^1|^{p-1} sgn(u^1)\pi_{x_1}.
\end{align*}
With the same computation as in the proof of Proposition \ref{proprho2}, since $\text{div }\tilde{v}=0$ we get
\begin{align*}
    \frac{1}{p}\dt \|u^1(t)\|_{L^p}^p & \leq m^{-1}\z(t)\|u^1(t)\|^{p-1}_{L^p}\|\pi_{x_1}(t) \|_{L^p}\\
    & \leq m^{-1}\z(t)\|u(t)\|^{p-1}_{2,p}\|\nabla \pi(t) \|_{2,p}.
\end{align*}
And by Proposition \ref{propp} we obtain
\begin{align*}
    \frac{1}{p}\dt \|u^1(t)\|_{L^p}^p & \leq m^{-1}\ie(t) K(\|\nabla\rho(t)\|_{1,p})\|u(t)\|^{p-1}_{2,p}\|\Tilde{v}(t)\|^2_{2,p}.
\end{align*}
Now, observe that the same estimation as above applies to $u^2(t)$ and $u^3(t)$. Consequently, it yields
\begin{align} \label{P16-eq2}
    \frac{1}{p}\frac{d}{dt}\|u(t)\|_{L^p}^p \leq 3m^{-1} \ie(t) K(\|\nabla\rho(t)\|_{1,p}) \|u(t)\|^{p-1}_{2,p} \|\Tilde{v}(t)\|^2_{2,p}.
\end{align}
\\ 
\noindent{\it \underline{Step 3}: Estimate of $\displaystyle \frac{1}{p}\frac{d}{dt}\|\nabla u(t)\|_{L^p}^p$.} Now, regarding the first derivative of each component of $u$, we have:
\begin{equation*}
    \dt u^1_{x_i}=- \ie\Tilde{v}_{x_i}\cdot \nabla u^1-\ie \Tilde{v}\cdot \nabla u^1_{x_i}+\z\rho^{-2} \rho_{x_i} \pi_{x_1}- \z\rho^{-1} \pi_{x_1 x_i} \,,
\end{equation*}
which let us deduce that
\begin{align*}
    \frac{1}{p}\dt|u^1_{x_i}|^p&= |u^1_{x_i}|^{p-1} sgn(u^1_{x_i}) \dt u^1_\xxi\\
    &=-\ie |u^1_{x_i}|^{p-1} sgn(u^1_{x_i}) \Tilde{v}_{x_i}\cdot \nabla u^1 -\ie |u^1_{x_i}|^{p-1} sgn(u^1_{x_i}) \Tilde{v} \cdot \nabla u^1_{x_i}\\
    &\quad+ \frac{\z}{\rho^2}\ |u^1_{x_i}|^{p-1} sgn (u^1_{x_i}) \, \rho_{x_i} \ \pi_{x_1}
    -\frac{\z}{\rho} |u^1_{x_i}|^{p-1} sgn(u^1_{x_i})  \ \pi_{x_1 x_i}.
\end{align*}
Then,
\begin{align}\label{P7-eq2}
     \frac{1}{p} \dt\|u^1_{x_i}(t)\|^p_{L^p}& = -\ie\int_{\RR^3} |u^1_{x_i}|^{p-1} sgn(u^1_{x_i}) \Tilde{v}_{x_i}\cdot \nabla u^1 \ dx\nonumber\\
     & -\ie \int_{\RR^3} |u^1_{x_i}|^{p-1} sgn(u^1_{x_i}) \Tilde{v} \cdot \nabla u^1_{x_i}  \ dx\nonumber\\
     & + \z \int_{\RR^3} |u^1_{x_i}|^{p-1} sgn (u^1_{x_i}) \, \rho^{-2} \ \rho_{x_i} \ \pi_{x_1} \ dx\nonumber\\
     & -\z \int_{\RR^3} |u^1_{x_i}|^{p-1} sgn(u^1_{x_i})  \ \rho^{-1} \ \pi_{x_1 x_i} \ dx \nonumber\\
     & = A_1 + A_2+A_3+A_4\,.
  \end{align}
Since $div \ \Tilde{v}=0$, we have $A_2=0$. We also see that
\begin{align*}
A_1 & \leq \ie\int_{\RR^3} |u^1_{x_i}|^{p-1} |\Tilde{v}_{x_i}| \  |\nabla u^1 | \ dx \leq \ie \sup_{x\in\RR^3} |\Tilde{v}_{x_i}|\int_{\RR^3}|u^1_{x_i}|^{p-1} | \nabla u^1| \,dx\\
& \leq \ie \|\Tilde{v}(t)\|_{2,p}\|u^1_{x_i}(t)\|^{p-1}_{L^p} \|\nabla u^1(t)\|_{L^p}\\
& \leq \ie \|\Tilde{v}(t)\|_{2,p}\|u(t)\|^{p}_{2,p}. 
\end{align*}
With the same idea as before and considering the Proposition \ref{propp}, we obtain
\begin{align*}
A_3 & \leq \z m^{-2} \int_{\RR^3} |u^1_{x_i}|^{p-1} \ |\rho_{x_i}| \ |\pi_{x_1}| \ dx \leq \z m^{-2} \sup_{x\in\RR^3 }|\rho_{x_i}|\int_{\RR^3}|u^1_{x_i}|^{p-1} \ |\pi_{x_1}|\,dx\\
& \leq \z m^{-2}\|\nabla \rho(t)\|_{1,p}\|u^1_{x_i}(t)\|^{p-1}_{L^p}\|\pi_{x_1}(t)\|_{L^p}\leq \z m^{-2}\|\nabla \rho(t)\|_{1,p}\|u(t)\|^{p-1}_{p,2}\|\nabla \pi(t)\|_{2,p}\\
& \leq \ie m^{-2} K(\|\nabla\rho(t)\|_{1,p}) \|\nabla \rho(t)\|_{1,p}\|u(t)\|^{p-1}_{p,2} \|\Tilde{v}(t)\|^2_{2,p}\,,
\end{align*}
and
\begin{align*}
A_4 & \leq \z m^{-1}\int_{\RR^3} |u^1_{x_i}|^{p-1} \ |\pi_{x_1 x_i}| \ dx\leq \z m^{-1} \|u^1_{x_i}(t)\|^{p-1}_{L^p}\|\pi_{x_1 x_i}(t)\|_{L^p}\\
& \leq \z m^{-1} \|u(t)\|^{p-1}_{p,2}\|\nabla \pi(t)\|_{2,p}\\
& \leq \ie m^{-1} K(\|\nabla\rho(t)\|_{1,p}) \|u(t)\|^{p-1}_{p,2} \|\Tilde{v}(t)\|^2_{2,p}\,.
\end{align*}
Thereby, by replacing the estimates of $A_1, A_2, A_3$ and $A_4$ in \eqref{P7-eq2}, we get
\begin{align*}
\frac{1}{p} \dt\|u^1_{x_i}(t)\|^p_{L^p} & \leq \ie \|\Tilde{v}(t)\|_{2,p}\|u(t)\|^{p}_{2,p}\\
& +\ie L_1(\|\nabla\rho(t)\|_{1,p})\|u(t)\|^{p-1}_{p,2} \|\Tilde{v}(t)\|^2_{2,p}\,,
\end{align*}
where
$$\Tilde{L}_1(\|\nabla\rho(t)\|_{1,p})=K(\|\nabla\rho(t)\|_{1,p})(m^{-2}\|\nabla \rho(t)\|_{1,p}+m^{-1})\,.$$
It is clear that the above estimate also holds to $u^2_{x_i}(t)$ and $u^3_{x_i}(t)$. Then,
\begin{align}\label{P17-eq1}
\frac{1}{p}\frac{d}{dt}\|\nabla u(t)\|_{L^p}^p & \leq C \ie \left[\|\Tilde{v}(t)\|_{2,p}\|u(t)\|^{p}_{2,p}\right.\nonumber\\
& +\left. L_1(\|\nabla\rho(t)\|_{1,p})\|u(t)\|^{p-1}_{p,2} \|\Tilde{v}(t)\|^2_{2,p}\right]\,.
\end{align}
\\ 
\noindent{\it \underline{Step 4}: Estimate of $\displaystyle \frac{1}{p}\frac{d}{dt}\|D^2 u(t)\|_{L^p}^p$.} Once again, by doing the calculations for the second derivative componentwise, we get:
\begin{align*}
      \frac{1}{p}\dt |u^1_{x_i x_j}|^p & =|u^1_\xij|^{p-1}sgn(u^1_\xij)\dt u^1_\xij\\
      & =-\ie |u^1_\xij|^{p-1}sgn(u^1_\xij) \left(-\Tilde{v}_{x_i x_j}\cdot \nabla u^1 + \Tilde{v}_{x_i}\cdot\nabla u^1_{x_j}+\Tilde{v}_{x_j}\cdot \nabla u^1_{x_i} +\Tilde{v}\cdot \nabla u^1_{x_i x_j}\right)\\
      &\quad+ \z|u^1_\xij|^{p-1}sgn(u^1_\xij)\left(-\frac{2}{\rho^3}\rho_{x_i}\rho_{x_j} \pi_{x_1}+\frac{1}{\rho^2}\rho_{x_i x_j} \pi_{x_1}
      +\frac{1}{\rho^2}\rho_{x_i} \pi_{x_1 x_j}\right.\\
      &\quad \left.+\frac{1}{\rho^2}\rho_{x_j}\pi_{x_1 x_i}-\frac{1}{\rho}\pi_{x_1 x_i x_j}\right).
\end{align*}
When integrating over $\mathbb{R}^3$, it yields
\begin{align}\label{P11-eq1}
\frac{1}{p}\dt\|u^1_\xij(t)\|_{L^p}^p & \leq -\ie\int_{\mathbb{R}^3} |u^1_\xij|^{p-1}sgn(u^1_\xij) \Tilde{v}_{x_i x_j}\cdot \nabla u^1 \ dx\nonumber\\ 
& \quad -\ie \int_{\mathbb{R}^3} |u^1_\xij|^{p-1}sgn(u^1_\xij) \Tilde{v}_{x_i}\cdot\nabla u^1_{x_j} \ dx \nonumber\\ 
& \quad -\ie \int_{\mathbb{R}^3} |u^1_\xij|^{p-1}sgn(u^1_\xij) \Tilde{v}_{x_j}\cdot \nabla u^1_{x_i} \ dx\nonumber\\ 
& \quad -\ie \int_{\mathbb{R}^3} |u^1_\xij|^{p-1}sgn(u^1_\xij) \Tilde{v}\cdot \nabla u^1_{x_i x_j} \ dx\nonumber\\ 
& \quad -2 \z  \int_{\mathbb{R}^3} |u^1_\xij|^{p-1}sgn(u^1_\xij) \frac{1}{\rho^3}\rho_{x_i}\rho_{x_j} \pi_{x_1} \ dx\nonumber\\ 
& \quad +\z \int_{\mathbb{R}^3} |u^1_\xij|^{p-1}sgn(u^1_\xij)\frac{1}{\rho^2}\rho_{x_i x_j} \pi_{x_1} \ dx \nonumber\\ 
& \quad +\z \int_{\mathbb{R}^3} |u^1_\xij|^{p-1}sgn(u^1_\xij) \frac{1}{\rho^2}\rho_{x_i} \pi_{x_1 x_j} \ dx\nonumber\\ 
& \quad +\z \int_{\mathbb{R}^3} |u^1_\xij|^{p-1}sgn(u^1_\xij) \frac{1}{\rho^2}\rho_{x_j}\pi_{x_1 x_i} \ dx\nonumber\\ 
& \quad -\z \int_{\mathbb{R}^3} |u^1_\xij|^{p-1}sgn(u^1_\xij) \frac{1}{\rho}\pi_{x_1 x_i x_j} \ dx\nonumber\\ 
& = B_1+B_2+B_3+B_4+B_5+B_6+B_7+B_8+B_9\,.
\end{align}
It is not difficult to show that $B_4=0$ because $div \ \Tilde{v}=0$. It is also easy to see that:
\begin{align*}
B_1 & \leq \ie\int_{\mathbb{R}^3} |u^1_\xij|^{p-1} |\Tilde{v}_{x_i x_j}|  \  |\nabla u^1| \ dx \leq \ie \supx|\nabla u^1|\intr |u^1_\xij|^{p-1} |\Tilde{v}_{x_i x_j}| \,dx\\
& \leq \ie \|\Tilde{v}_\xij(t)\|_{L^p}\|u(t)\|^p_{2,p}\leq \ie \|\Tilde{v}(t)\|_{2,p}\|u(t)\|^p_{2,p}.
\end{align*}
Also,
\begin{align*}
B_2 & \leq \ie \int_{\mathbb{R}^3} |u^1_\xij|^{p-1} |\Tilde{v}_{x_i}| \ |\nabla u^1_{x_j}| \ dx \leq \ \ie\supx|\Tilde{v}_{x_i}|\intr |u^1_\xij|^{p-1}|\nabla u^1_{x_j}|\,dx\\
& \leq  \ \ie\|\Tilde{v}(t)\|_{2,p}\|u(t)\|^p_{2,p}.
\end{align*}
The estimation of $B_3$ is the same of $B_2$. Now, see that
\begin{align*}
B_5 & \leq 2 \z m^{-3} \int_{\mathbb{R}^3} |u^1_\xij|^{p-1} |\rho_{x_i}| \ |\rho_{x_j}| \  |\pi_{x_1}| \ dx\\
& \leq 2 \z m^{-3} \left(\supx|\rho_{x_i}|\right)\left(\supx|\rho_{x_j}|\right)\intr |u^1_\xij|^{p-1}|\pi_{x_1}|\,dx\\
& \leq 2 \z m^{-3} \|\nabla \rho(t)\|^2_{1,p}\|u^1_\xij(t)\|^{p-1}_{L^p} \|\pi_{x_1}(t)\|_{L^p}\\
& \leq 2 \z m^{-3} \|\nabla \rho(t)\|^2_{1,p}\|u(t)\|^{p-1}_{2,p} \|\nabla \pi(t)\|_{2,p}.
\end{align*}
By applying the Proposition \ref{propp} we have
\begin{align*}
B_5 & \leq 2 \ \ie m^{-3} K(\|\nabla\rho(t)\|_{1,p}) \|\nabla \rho(t)\|^2_{1,p}\|u(t)\|^{p-1}_{2,p} \|\Tilde{v}(t)\|^2_{2,p}.
\end{align*}
Using again the Proposition \ref{propp} we can be estimated the term $B_6$ by
\begin{align*}
B_6 & \leq \z m^{-2}\int_{\mathbb{R}^3} |u^1_\xij|^{p-1} |\rho_{x_i x_j}| |\pi_{x_1}| \ dx \leq \z m^{-2} \supx |p_{x_1}|\intr |u^1_\xij|^{p-1}|\rho_{x_i x_j}|\,dx\\
& \leq \z m^{-2} \|\nabla \pi(t)\|_{2,p} \|u^1_\xij(t)\|^{p-1}_{L^p} \ \|\rho_{x_i x_j}(t)\|_{L^p}\\
& \leq \z m^{-2} \|\nabla \pi(t)\|_{2,p} \|u(t)\|^{p-1}_{2,p} \ \|\nabla\rho(t)\|_{1,p}\\
& \leq  \ie m^{-2} K(\|\nabla\rho(t)\|_{1,p}) \ \|\nabla\rho(t)\|_{1,p} \ \|u(t)\|^{p-1}_{2,p} \|\Tilde{v}(t)\|^2_{2,p}\,.
\end{align*}
Repeating the procedure above we have
\begin{align*}
B_7 & \leq \z m^{-2}\int_{\mathbb{R}^3} |u^1_\xij|^{p-1} |\rho_{x_i}| \ |\pi_{x_1 x_j}| \ dx \leq \z m^{-2} \supx |\rho_{x_i}|\intr |u^1_\xij|^{p-1}|\pi_{x_1 x_j}|\,dx\\
& \leq \z m^{-2} \ \|\nabla\rho(t)\|_{1,p} \|u(t)\|^{p-1}_{2,p} \ \|\nabla \pi(t)\|_{2,p}\\
& \leq \ \ie m^{-2} K(\|\nabla\rho(t)\|_{1,p}) \ \|\nabla\rho(t)\|_{1,p} \ \|u(t)\|^{p-1}_{2,p} \|\Tilde{v}(t)\|^2_{2,p}\,,
\end{align*}
and
\begin{align*}
B_8 & \leq \ \ie m^{-2} K(\|\nabla\rho(t)\|_{1,p}) \ \|\nabla\rho(t)\|_{1,p} \ \|u(t)\|^{p-1}_{2,p} \|\Tilde{v}(t)\|^2_{2,p}\,.
\end{align*}
Finally, observe that
\begin{align*}
B_9 & \leq \z m^{-1} \int_{\mathbb{R}^3} |u^1_\xij|^{p-1} |\pi_{x_1 x_i x_j}| \ dx \leq \z m^{-1}  \|u(t)\|^{p-1}_{2,p} \ \|\nabla \pi(t)\|_{2,p}\\
& \leq \ie m^{-1} K(\|\nabla\rho(t)\|_{1,p})  \|u(t)\|^{p-1}_{2,p} \ \|\Tilde{v}(t)\|^2_{2,p}.
\end{align*}
By substituting the estimates of $B_1, B_2, B_3, B_4, B_5, B_6, B_7, B_8$ and $B_9$ in \eqref{P11-eq1} we obtain 
\begin{align*}
\frac{1}{p}\dt\|u^1_\xij(t)\|_{L^p}^p & \leq 3\ie \|\Tilde{v}(t)\|_{2,p}\|u(t)\|^p_{2,p}\\
& + 2 \ie m^{-3} K(\|\nabla\rho(t)\|_{1,p}) \|\nabla \rho(t)\|^2_{1,p}\|u(t)\|^{p-1}_{2,p} \|\Tilde{v}(t)\|^2_{2,p}\\
& +3\ie m^{-2} K(\|\nabla\rho(t)\|_{1,p}) \ \|\nabla\rho(t)\|_{1,p} \ \|u(t)\|^{p-1}_{2,p} \|\Tilde{v}(t)\|^2_{2,p}\\
& +\ie m^{-1} K(\|\nabla\rho(t)\|_{1,p})  \|u(t)\|^{p-1}_{2,p} \ \|\Tilde{v}(t)\|^2_{2,p}\\
& = 3 \ie \|\Tilde{v}(t)\|_{2,p}\|u(t)\|^p_{2,p}+\ie \Tilde{L}_2(\|\nabla\rho(t)\|_{1,p}) \|u(t)\|^{p-1}_{2,p} \|\Tilde{v}(t)\|^2_{2,p}\,,
\end{align*} 
where
\begin{align*}
\Tilde{L}_2(\|\nabla\rho(t)\|_{1,p}) & =K(\|\nabla\rho(t)\|_{1,p})\Big[2 m^{-3} \|\nabla \rho(t)\|^2_{1,p}+3 m^{-2}  \|\nabla\rho(t)\|_{1,p}+m^{-1}\Big].
\end{align*}
One more time, we see that the above estimate holds also to $u^2_\xij(t)$ and $u^3_\xij(t)$. Therefore, 
\begin{align}\label{P16-eq3}
\frac{1}{p}\frac{d}{dt}\|D^2 u(t)\|_{L^p}^p & \leq C\ie\left(3 \|\Tilde{v}(t)\|_{2,p}\|u(t)\|^p_{2,p}+\Tilde{L}_2(\|\nabla\rho(t)\|_{1,p}) \|u(t)\|^{p-1}_{2,p} \|\Tilde{v}(t)\|^2_{2,p}\right)\,.
\end{align}
\\ 
\noindent{\it \underline{Step 5}: Conclusion.} The steps 2, 3 and 4 yield
\begin{align*}
\frac{1}{p}\frac{d}{dt}\|u(t)\|^p_{2,p} & \leq 3m^{-1} \ie(t) K(\|\nabla\rho(t)\|_{1,p}) \|u(t)\|^{p-1}_{2,p} \|\Tilde{v}(t)\|^2_{2,p}\\
& +C \ie \left[\|\Tilde{v}(t)\|_{2,p}\|u(t)\|^{p}_{2,p}+\Tilde{L}_1(\|\nabla\rho(t)\|_{1,p})\|u(t)\|^{p-1}_{p,2} \|\Tilde{v}(t)\|^2_{2,p}\right]\\
& +C\ie\left(3 \|\Tilde{v}(t)\|_{2,p}\|u(t)\|^p_{2,p}+\Tilde{L}_2(\|\nabla\rho(t)\|_{1,p}) \|u(t)\|^{p-1}_{2,p} \|\Tilde{v}(t)\|^2_{2,p}\right)\\
& \leq C \ie \|\Tilde{v}(t)\|_{2,p}\|u(t)\|^{p}_{2,p} +C\ie L(\|\nabla\rho(t)\|_{1,p}) \|u(t)\|^{p-1}_{2,p} \|\Tilde{v}(t)\|^2_{2,p}\,,
\end{align*}
with $L$ defined by
\begin{align*}
L(\|\nabla\rho(t)\|_{1,p})=3m^{-1}  K(\|\nabla\rho(t)\|_{1,p})+\Tilde{L}_1(\|\nabla\rho(t)\|_{1,p})+\Tilde{L}_2(\|\nabla\rho(t)\|_{1,p}).
\end{align*}
Therefore, 
\begin{align*}
\frac{d}{dt}\|u(t)\|_{2,p}\leq c_2 \left[\ie \|\Tilde{v}(t)\|_{2,p}\|u(t)\|_{2,p}+\ie L(\|\nabla\rho(t)\|_{1,p})  \|\Tilde{v}(t)\|^2_{2,p}\right].
\end{align*}
By Gronwall's inequality our desired estimate holds.

\end{proof}

\begin{proposition}\label{propubi}
   Let $\rho$ and $\nabla \pi$ be the unique solutions described in Propositions \ref{proprho} and \ref{proppbi}, respectively. Also, let $v \in C([0,T];W^{2,p}(\RR^3))$ with $\text{div }v=0$. Then, the problem
\begin{equation}\label{eq3bi}
\begin{cases}
 u_t + (v\cdot \nabla) u  
 + (v\cdot \nabla) \W^{Q} 
 + \frac{\nabla \pi}{\rho}=0 \\
u|_{t=0} =v_0(x), &
\end{cases}
\end{equation}
has a unique solution $u\in C([0,T];W^{2,p}(\RR^3))$. Moreover, $u$ satisfies
\[
\|u(t)\|_{2,p}\leq \text{exp}\left( c_2\intt \|v(s)\|_{2,p}  ds\right) \times
 \]   
   
    \[
    \left[\|v_0\|_{2,p}+ \intt  L(\|\nabla \rho(s)\|_{1,p})\| v(s)\|^2_{2,p}  + \| v(s)\|_{2,p}\|\W(s)^{Q}\|_{k,2} \,ds\right].
\]

\end{proposition}

\begin{proof}The proof follows the same ideas as proposition  \ref{propu}, we omit  some parts of the proof. For simplicity, we denote $\W^Q$ instead of $\W^Q_i$ for $i=1,2,3.$

\noindent{\it \underline{Step 1}: Transport equation.} If we consider $u=(u^1, u^2, u^3)$ we can rewritten the Eq. \eqref{eq3bi} componentwise by:
\begin{align*}
\left\{
\begin{aligned}
& u^{i}_t + v\cdot \nabla u^i
 + (v\cdot \nabla) \W^{Q}
=- \  \frac{ \pi_{x_i}}{\rho}      \\
& u^i|_{t=0} =v^i_0(x), 
\end{aligned}
\,,\quad \text{ with }\quad i=1,2,3.
\right.
\end{align*}

\noindent{\it \underline{Step 2}:   We will estimated $\displaystyle \frac{1}{p}\frac{d}{dt}\|u(t)\|_{L^p}^p$.} Here we also proceed componentwise. Thus, for $i=1$, we obtain,

With the  similar calculation  as in the proof of Proposition \ref{propu}, we deduce that 

\begin{align*}
    \frac{1}{p}\dt \|u^1(t)\|_{L^p}^p & \leq m^{-1}  \|u^1(t)\|^{p-1}_{L^p}\|\pi_{x_1}(t) \|_{L^p}  
    + \|u^1(t)\|^{p-1}_{L^p} \|v(t)\|_{L^p}  \|\W^Q\|_{k,2}
  \\  & 
  \leq  m^{-1}  \|u^1(t)\|^{p-1}_{L^p}\|\pi_{x_1}(t) \|_{2,p}  
    + \|u^1(t)\|^{p-1}_{L^p} \|v(t)\|_{L^p}  \|\W^Q\|_{k,2}.
\end{align*}
And by Proposition \ref{proppbi} we obtain
\begin{align*}
    \frac{1}{p}\dt \|u^1(t)\|_{L^p}^p & \leq m^{-1} K(\|\nabla\rho(t)\|_{1,p})\|u(t)\|^{p-1}_{2,p} \|v(t)\|^2_{2,p}  
    \\  & 
    +  \|u^1(t)\|^{p-1}_{L^p} \|v(t)\|_{L^p}  \|\W^Q\|_{k,2} .
\end{align*}
Now, observe that the same estimation as above applies to $u^2(t)$ and $u^3(t)$. Then we have

\begin{align} \label{P16-eq2bi}
    \frac{1}{p}\frac{d}{dt}\|u(t)\|_{L^p}^p  
\leq m^{-1} K(\|\nabla\rho(t)\|_{1,p})\|u(t)\|^{p-1}_{2,p} \|v(t)\|^2_{2,p}  \nonumber
    \\  + \|u(t)\|^{p-1}_{L^p} \|v(t)\|_{L^p}  \|\W^Q\|_{k,2}.
\end{align}
\\ 
\noindent{\it \underline{Step 3}: Estimate of $\displaystyle \frac{1}{p}\frac{d}{dt}\|\nabla u(t)\|_{L^p}^p$.} Now, regarding the first derivative of each component of $u$, we obtain :

\[
    \dt u^1_{x_i}=- v_{x_i}\cdot \nabla u^1- v\cdot \nabla u^1_{x_i}+ \rho^{-2} \rho_{x_i} \pi_{x_1}- \rho^{-1} \pi_{x_1 x_i} 
\]

  \[
     - v_{x_i}  \cdot \nabla \W^{Q} 
     - v \cdot \nabla \W^{Q}_{x_i} \,,
 \]

thus we have 
\begin{align*}
    \frac{1}{p}\dt|u^1_{x_i}|^p&= |u^1_{x_i}|^{p-1} sgn(u^1_{x_i}) \dt u^1_\xxi\\
    &=- |u^1_{x_i}|^{p-1} sgn(u^1_{x_i}) v_{x_i}\cdot \nabla u^1 -|u^1_{x_i}|^{p-1} sgn(u^1_{x_i}) v \cdot \nabla u^1_{x_i}\\
    &\quad+ \frac{1}{\rho^2}\ |u^1_{x_i}|^{p-1} sgn (u^1_{x_i}) \, \rho_{x_i} \ \pi_{x_1}
    -\frac{1}{\rho} |u^1_{x_i}|^{p-1} sgn(u^1_{x_i})  \ \pi_{x_1 x_i}\\
    &\quad  -   |u^1_{x_i}|^{p-1} sgn(u^1_{x_i})  v_{x_i} \cdot \nabla \W^{Q}
    - |u^1_{x_i}|^{p-1} sgn(u^1_{x_i})   v \cdot \nabla \W^{Q}_{x_i}.  
  \end{align*}

Therefore, we deduce 
\begin{align}\label{P7-eq2bi}
     \frac{1}{p} \dt\|u^1_{x_i}(t)\|^p_{L^p}& = -\int_{\RR^3} |u^1_{x_i}|^{p-1} sgn(u^1_{x_i}) v_{x_i}\cdot \nabla u^1 \ dx\nonumber\\
     & - \int_{\RR^3} |u^1_{x_i}|^{p-1} sgn(u^1_{x_i}) v \cdot \nabla u^1_{x_i}  \ dx\nonumber\\
     & +  \int_{\RR^3} |u^1_{x_i}|^{p-1} sgn (u^1_{x_i}) \, \rho^{-2} \ \rho_{x_i} \ \pi_{x_1} \ dx\nonumber\\
     & -\int_{\RR^3} |u^1_{x_i}|^{p-1} sgn(u^1_{x_i})  \ \rho^{-1} \ \pi_{x_1 x_i} \ dx \nonumber
     \\
    &\quad  - \int_{\RR^3}  |u^1_{x_i}|^{p-1} sgn(u^1_{x_i})  v_{x_i} \cdot \nabla \W^{Q} dx\nonumber
   \\ &
    - \int_{\RR^3} |u^1_{x_i}|^{p-1} sgn(u^1_{x_i})  v \cdot \nabla \W^{Q}_{x_i} dx \nonumber
    \\ & =  \sum_{i=1}^{6}  A_i \,.
  \end{align}

 Arguing  as in step 3 of Proposition \ref{propu} we have 

 \[
 A_{2}=0, 
 \]

\begin{align*}
A_1 & \leq  \|v(t)\|_{2,p}\|u(t)\|^{p}_{2,p},
\end{align*}

\begin{align*}
A_3 & \leq  m^{-2} K(\|\nabla\rho(t)\|_{1,p}) \|\nabla \rho(t)\|_{1,p} \|u(t)\|^{p-1}_{p,2} \|v(t)\|^2_{2,p} \,,
\end{align*}
and

\begin{align*}
A_4 & \leq  m^{-1} K(\|\nabla\rho(t)\|_{1,p}) \|u(t)\|^{p-1}_{p,2}  \|v(t)\|^2_{2,p} \,.
\end{align*}

Now, by H\"older’s inequality, we have 

\[
A_{5}\leq |\int_{\RR^3}  |u^1_{x_i}|^{p-1} sgn(u^1_{x_i})  v_{x_i} \cdot \nabla \W^{Q} dx|
\]

\[
\leq  \|u(t)\|^{p-1}_{2,p}  \|v(t)\|_{2,p}  \|\W^Q \|_{k,2}, 
\]

\[
 A_{6} \leq |\int_{\RR^3} |u^1_{x_i}|^{p-1} sgn(u^1_{x_i})   v \cdot \nabla \W^{Q}_{x_i} dx|
\]

\[
 \leq \|u(t)\|^{p-1}_{2,p}  \|v(t)\|_{2,p}  \|\W^Q \|_{k,2}. 
\]

From (\ref{P7-eq2bi})   and the estimations for $A_{i}$  we deduce  that 

\begin{align*}
\frac{1}{p} \dt\|u^1_{x_i}(t)\|^p_{L^p} & \leq \|v(t)\|_{2,p}\|u(t)\|^{p}_{2,p} + 
\|u(t)\|^{p-1}_{2,p}  \|v(t)\|_{2,p}  \|\W^Q \|_{k,2}. 
\\ & +  \Tilde{L_1}(\|\nabla\rho(t)\|_{1,p})\|u(t)\|^{p-1}_{p,2} \|v(t)\|^2_{2,p} \,,
\end{align*}

where
$$\Tilde{L_1}(\|\nabla\rho(t)\|_{1,p})=K(\|\nabla\rho(t)\|_{1,p})(m^{-2}\|\nabla \rho(t)\|_{1,p}+m^{-1})\,.$$
It is clear that the above estimate also holds to $u^2_{x_i}(t)$ and $u^3_{x_i}(t)$. Then,
\begin{align}\label{P17-eq1bi}
\frac{1}{p}\frac{d}{dt}\|\nabla u(t)\|_{L^p}^p & \leq C  \left[\|v(t)\|_{2,p}\|u(t)\|^{p}_{2,p}  + \|u(t)\|^{p-1}_{2,p}  \|v(t)\|_{2,p}  \|\W^Q \|_{k,2}.  \right.\nonumber\\
& +\left. L_1(\|\nabla\rho(t)\|_{1,p})\|u(t)\|^{p-1}_{p,2} \|v(t)\|^2_{2,p}    \right]\,.
\end{align}

\noindent{\it \underline{Step 4}: Estimate of $\displaystyle \frac{1}{p}\frac{d}{dt}\|D^2 u(t)\|_{L^p}^p$.} Following the same kind of estimation as in step 4 of the proposition,  we have

\begin{align}\label{P16-eq3bi}
\frac{1}{p}\frac{d}{dt}\|D^2 u(t)\|_{L^p}^p & \leq C (3 \|v(t)\|_{2,p}\|u(t)\|^p_{2,p} + \|u(t)\|^{p-1}_{2,p} \|v(t)\|_{2,p}  \|\W^Q \|_{k,2}   \\ &  +\Tilde{L}_2(\|\nabla\rho(t)\|_{1,p}) \|u(t)\|^{p-1}_{2,p} \|v(t)\|^2_{2,p}  ) \,.
\end{align}

where
\begin{align*}
\Tilde{L}_2(\|\nabla\rho(t)\|_{1,p}) & =K(\|\nabla\rho(t)\|_{1,p})\Big[2 m^{-3} \|\nabla \rho(t)\|^2_{1,p}+3 m^{-2}  \|\nabla\rho(t)\|_{1,p}+m^{-1}\Big].
\end{align*}

\noindent{\it \underline{Step 5}: Conclusion.}  From the steps 2, 3 and 4 we conclude

\begin{align*}
\frac{d}{dt}\|u(t)\|_{2,p}\leq c_2 \left[\|v(t)\|_{2,p}\|u(t)\|_{2,p} + \|v(t)\|_{2,p} \|\W^Q \|_{k,2}  
  \right.\\
  +\left. L(\|\nabla\rho(t)\|_{1,p})
\|v(t)\|^2_{2,p}\right].
\end{align*}

with $L$ defined by
\begin{align*}
L(\|\nabla\rho(t)\|_{1,p})=3m^{-1}  K(\|\nabla\rho(t)\|_{1,p})+\Tilde{L}_1(\|\nabla\rho(t)\|_{1,p})+\Tilde{L}_2(\|\nabla\rho(t)\|_{1,p}).
\end{align*}
By Gronwall's inequality our desired estimate holds.

\end{proof}


\section{Proof of the theorem \ref{maintheorem} } 

The main goal of this section is to prove Theorem \ref{maintheorem}. To this end, we will conveniently use the auxiliary problems described in the preceding sections, along with their corresponding estimates.   
 The proof of the existence of solutions will be divided into several steps and will ultimately be concluded in Proposition \ref{existence}. The central idea behind this proof is the construction of a suitable sequence of successive approximations, iteratively obtained as solutions to our auxiliary problems. 
 The uniqueness of solutions is established in Proposition \ref{uniqueness} and relies on some of the ideas previously introduced in the proof of the existence theorem.

\subsection{Successive approximations}
Now, we explicitly describe the approximations announced before. 
Let $\Tilde{v}^{(0)}=0$, and for $k\in \NN$, let  $\rho^{(k)}$, $\nabla \pi^{(k)}$ and $u^{(k)}$ be the corresponding solutions of the following problems
\begin{equation}\label{iterho}
\begin{cases} 
\rho^{(k)}_t + \ie \Tilde{v}^{(k-1)} \cdot \nabla \rho^{(k)} =0, &\\ 
\rho^{(k)}|_{t=0}=\rho_0(x),&
\end{cases}
\end{equation}
\begin{equation}\label{itep}
\begin{cases} 
\text{div}\left(\frac{1}{\rho^{(k)}}\nabla \pi^{(k)}\right)=-\z^{-2}\sum_{i, j=1}^3  \Tilde{v}^{(k-1),i}_{x_j}\Tilde{v}^{(k-1),j}_{x_i}, &\\ 
\end{cases}
\end{equation}

\begin{equation}\label{iteu}
\begin{cases}  
 u_t^{(k)} + \ie \Tilde{v}^{(k-1)}\cdot \nabla u^{(k)}+ \frac{\z}{\rho^{(k)}}\nabla \pi^{(k)}=0,\\
u^{(k)}|_{t=0} =v_0(x), &
\end{cases}
\end{equation}
Finally, we define 
\begin{equation}\label{defv}
    \Tilde{v}^{(k)}= u^{(k)}-\nabla \phi^{(k)}
\end{equation}
where $\phi^{(k)}$is the solution of 
\begin{equation}\label{eqphi}
\Delta \phi^{(k)}=\text{div } u^{(k)}.
\end{equation}

The following result is crucial to prove that our successive approximations converge to the desired solution.

\begin{proposition}\label{propbouv}
There exists a stopping time $\tau$, such that the sequence $\left\{\Tilde{v}^{(k)}\right\}_k$ is bounded in $C([0,\tau];W^{2,p}(\RR^3))$. In particular,
 \begin{equation*}
     \sup_{0\leq t \leq \tau}\|\Tilde{v}^{(k)}(t)\|_{2,p}\leq A,
\end{equation*} 
for some constant $A>1$.
\end{proposition}
\begin{proof}
Recalling the estimates of Proposition \ref{propu} we get
\begin{align*}
    \|u\kk(t)\|_{2,p}&\leq \text{exp}\left( c_2\intt \ie(s)\|\Tilde{v}\ek(s)\|_{2,p}\,ds\right)\\&\hspace{0.5cm}\cdot\left[\|v_0\|_{2,p}+ \intt \ie(s) L(\|\nabla \rho\kk(s)\|_{1,p})\| \Tilde{v}\ek(s)\|^2_{2,p}\,ds\right].
\end{align*}
By Eq.\eqref{defv} and Eq.\eqref{eqphi} we have
\begin{equation*}
\|\Tilde{v}^{(k)}\|_{2,p}\leq \|u^{(k)}\|_{2,p} + \|\nabla \phi^{(k)}\|_{2,p}\leq c_3\|u^{(k)}\|_{2,p},
\end{equation*}
which combined with the above estimate implies
\begin{align*}\label{inductionk}
\|\Tilde{v}^{(k)}(t)\|_{2,p}&\leq c_3 \text{ exp}\left( c_2\intt \ie(s) \|\Tilde{v}\ek(s)\|_{2,p}\,ds\right)\\
\notag &\hspace{1cm}\cdot\big[\|v_0\|_{2,p}+ \intt \ie(s) L(\|\nabla \rho\kk(s)\|_{1,p})\|\Tilde{v}\ek(s)\|^2_{1,p}\,ds\big].
\end{align*}
Now, we consider $A>1$ with 
\begin{align*}
    A\geq c_3\text{exp}(c_2)\left[\|v_0\|_{2,p}+L\left(\|\nabla \rho_0\|_{1,p}\text{exp}(c_1)\right) \right].
\end{align*}
We define 
 \begin{equation*}
        \tau=\inf\{t\geq 0 : \intt \ie(s)\,ds\geq A^{-2}\}\wedge T,
  \end{equation*}
    \ is a stopping time following \cite[Proposition 3.9]{L}.
 Now, we can prove our claim by induction on $k$. The case $k=0$ trivially holds.
 If  we suppose
\begin{equation*}
     \sup_{0\leq t \leq \tau}\|\Tilde{v}^{(k-1)}\|_{2,p}\leq A,
\end{equation*}
then using  Proposition \ref{proprho2} combined with the fact $\|\Tilde{v}(t)\|_{2,p}=\z(t)\|v(t)\|_{2,p}$, we obtain the following estimate for every $t\leq \tau$:
\begin{align*}
    \|\nabla\rho\kk(t)\|_{1,p}& \leq \|\nabla\rho_0\|_{1,p}\text{ exp}\left(c_1 \intt\ie(s)\|\Tilde{v}\ek(s)\|_{2,p} \,ds\right)\\
    &\leq  \|\nabla\rho_0\|_{1,p}\text{ exp}\left(c_1 A \intt\ie(s)\,ds\right)\\
    &\leq  \|\nabla\rho_0\|_{1,p}\text{ exp}\left(c_1\right),
\end{align*}
 where the last inequality follows from the definition of the stopping time $\tau$ and the assumption $\frac{1}{A^2}\leq \frac{1}{A}$. 
Finally, using the above estimates for $\|\Tilde{v}^{(k-1)}\|_{2,p}$ and $\|\nabla\rho\kk\|_{1,p}$, and the fact $L$ are increasing functions, we can deduce 
\begin{align*}
\|\Tilde{v}^{(k)}(t)\|_{2,p}&\leq c_3 \text{ exp}\left( c_2\intt \ie(s) A\,ds\right)\\
\notag &\hspace{1cm}\cdot\big[\|v_0\|_{2,p}+ \intt \ie(s) L(\|\nabla\rho_0\|_{1,p}\text{ exp}\left(c_1\right))A^2\,ds\big]\\
&\leq c_3 \text{ exp}\left(c_2\right)\cdot\big[\|v_0\|_{2,p}+  L(\|\nabla\rho_0\|_{1,p}\text{ exp}\left(c_1\right))\big]\leq A,
\end{align*}
completing our proof.
\end{proof}

As a direct consequence of the above proposition and its corresponding proof, we deduce the following bounds for $\{\rho^{(k)}\}_k$, and $\{\pi^{(k)}\}_k$.

\begin{corollary}\label{coroest}  The next estimates hold
\begin{equation}\label{esrho1}
    \sup_{0\leq t\leq \tau}\|\nabla\rho^{(k)}(t)\|_{1,p}\leq \text{exp}(c_1)\|\nabla\rho_0\|_{1,p}=L_1,
\end{equation}


\begin{equation}\label{esp1}
    \sup_{0\leq t\leq \tau}\|\nabla \pi^{(k)}(t)\|_{2,p}\leq \C^2 K(L_1)A^2=L_2,
\end{equation}

where $\C(\omega)=\sup_{0\leq t\leq \tau} \ie(t,\omega)$.
\end{corollary}


\subsection{Convergence results}

To show the convergence of the successive approximation constructed in the previous section, we consider the sequences of differences given by: $\sigma\kk=\rho\kk-\rho\ek$, $h\kk=u\kk-u\ek$, $q\kk=\pi\kk-\pi\ek $, $w\kk= v\kk- v\ek$ and $\Tilde{w}\kk= \Tilde{v}\kk- \Tilde{v}\ek$.

\begin{remark}
Subtracting the equations in each of the problems \eqref{iterho}, \eqref{itep} and \eqref{iteu} corresponding to the steps $(k)$ and $(k-1)$, we obtain 
\begin{equation}\label{defsigma}
\begin{cases} 
\sigma^{(k)}_t + \ie \Tilde{v}^{(k-1)} \cdot \nabla \sigma^{(k)} =-\ie \nabla \rho\ek \cdot \Tilde{w}\ek, &\\ 
\sigma^{(k)}|_{t=0}=0,&
\end{cases}
\end{equation}
\begin{align}\label{defq}
\begin{cases}
\text{div}\left(\frac{1}{\rho^{(k)}}\z\nabla q^{(k)}\right)&=\text{div}\left( \frac{\sigma\kk}{\rho\kk \rho\ek}\z\nabla \pi\ek \right)-\ie\sum_{i, j=1}^3   \Tilde{w}^{(k-1),i}_{x_j} \Tilde{v}^{(k-1),j}_{x_i}\\
&-\ie \sum_{i, j=1}^3  \Tilde{v}_\xxj^{(k-2),i}\Tilde{w}_\xxi^{(k-1),j},  
\end{cases}
\end{align}

\begin{align}\label{defh}
\begin{cases} 
 h_t^{(k)} + \ie \Tilde{v}^{(k-1)}\cdot \nabla h^{(k)} + \z\frac{\nabla q^{(k)}}{\rho^{(k)}}=\z\frac{\sigma\kk}{\rho\ek \rho\kk}\nabla \pi\ek - \ie \Tilde{w}\ek\nabla u\ek, \\
h^{(k)}|_{t=0} =0.
\end{cases}
\end{align}
\end{remark}

In the following three technical lemmas we deduce some estimates of the new variables $\sigma\kk$, $q\kk$ and $h\kk$, bounded in terms of  $\tilde{w}\kk$. The ideas behind the proofs are similar to the ones used when we proved the corresponding estimates for $\rho$, $p$, and $u$. 

\begin{lemma}\label{sigmaest}
    Let $\sigma\kk$ be the solution of the system \eqref{defsigma}, then 
    \begin{align}
    \|\sigma\kk(t)\|_{1,p}\leq 
    L_5 \intt \ie(s) \| \Tilde{w}\ek(s)\|_{1,p}\,ds,
    \end{align}  
for all $t\in [0, \tau]$.
\end{lemma}
\begin{proof}
We use  the same ideas of the proof of Proposition \ref{proprho}, first we obtain
\begin{align*}
  \frac{1}{p} \dt \|\sigma\kk(t)\|^p_{L^p} &\leq  \ie\supx |\Tilde{w}\ek| \intr | \sigma\kk|^{p-1}|\nabla \rho\ek|\, dx\\
   & \leq \ie \|\Tilde{w}\ek\|_{1,p}\|\sigma\kk\|_{L^p}^{p-1}\|\nabla\rho\ek\|_{L^p}\\
   &\leq \ie L_1 \|\Tilde{w}\ek\|_{1,p}\|\sigma\kk\|_{L^p}^{p-1}.
\end{align*}
Then, the same computation on the first derivative shows that
\begin{align*}
    \frac{1}{p}\dt \|\sigma_\xxi\kk(t)\|^p_{L^p}&\leq   \ie \bigg (\supx |\Tilde{v}_\xxi\ek|\intr |\sigma_\xxi\kk|^{p-1}|\nabla\sigma\kk|\,dx\\
    &\hspace{1cm}+ \supx |\nabla \rho\ek| \intr |\sigma_\xxi\kk|^{p-1}|\Tilde{w}_\xxi\ek|\,dx\\
    &\hspace{1cm}+  \supx |\Tilde{w}\ek| \intr |\sigma_\xxi\kk|^{p-1}|\nabla\rho_\xxi\ek|\,dx \bigg 
    )\\
    &\leq  \ie\left(A \|\nabla\sigma\kk\|_{L^p}^p+ 2  L_1 \|\Tilde{w}\ek \|_{1,p} \|\nabla\sigma\kk\|_{L^p}^{p-1}\right).
\end{align*}
Consequently, we deduce
\begin{align*}
    \dt \| \sigma\kk(t)\|_{1,p}\leq c_4  \ie \left( L_4  \|\Tilde{w}\ek(t) \|_{1,p}+ A \| \sigma\kk(t)\|_{1,p}\right),
\end{align*}
which by Gronwall's inequality implies our claim.
\end{proof}

\begin{lemma}\label{qest}
Let $q\kk$ be the solution of the system \eqref{defq}, then     
\begin{equation*}
    \|\nabla q\kk(t)\|_{1,p}\leq L_6\left(   \|\sigma\kk(t)\|_{1,p}+\z^{-2}(t)\|\Tilde{w}\ek(t)\|_{1,p}\right),
\end{equation*}
for all $t\in [0, \tau]$.
\end{lemma}
\begin{proof}
First, we define 
\begin{equation*}
I_1= \text{div}\left( \frac{\sigma\kk}{\rho\kk \rho\ek}\nabla \pi\ek \right),\\  
\end{equation*}
\begin{equation*}
I_2=-\z^{-2}\sum_{i, j=1}^3  \Tilde{w}^{(k-1),i}_{x_j} \Tilde{v}^{(k-1),j}_{x_i},  
\end{equation*}
\begin{equation*}
I_3= -\z^{-2}\sum_{i, j=1}^3  \Tilde{v}_\xxj^{(k-2),i}\Tilde{w}_\xxi^{(k-1),j}.
\end{equation*}
Then, using the same idea as in the proof of \cite[Lemma 3.1]{ITO}, we deduce the estimate
\begin{equation}\label{esI}
    \|\nabla q\kk\|_{1,p}\leq K_3\left(\bigg\|\frac{1}{\rho\kk}\bigg\|_{C^1(\RR^3)}\right) \|I_1 +I_2+ I_3\|_{L^p}.
\end{equation}
Now, we want to control the norms of each $I_1$, $I_2$ and $I_3$. Again, applying the same method as in the other estimates, we compute the following: 
\begin{align*}
  \|I_1\|_{L^p}&=\bigg\|\nabla\left(\frac{\sigma\kk}{\rho\kk \rho\ek}\right) \nabla \pi\ek+ \frac{\sigma\kk}{\rho\kk \rho\ek} \Delta \pi\ek\bigg\|_{L^p}\\
  &\leq \bigg\| \sum_{i=1}^3\frac{\sigma_\xxi\kk \rho\kk \rho\ek-\sigma\kk\left(\rho_\xxi\kk \rho\ek+\rho\kk\rho_\xxi\ek \right)}{(\rho\kk\rho\ek)^2}e_i\cdot \nabla \pi\ek\bigg \|_{L^p}\\
  &+ \bigg\|\frac{\sigma\kk}{\rho\kk \rho\ek}\Delta \pi\ek \bigg\|_{L^p}\\
  &\leq  m^{-2}\bigg(m^{-2}\big\| \left(\nabla\sigma\kk \rho\kk \rho\ek -\sigma\kk \nabla\rho\kk \rho\ek -\sigma\kk \rho\kk \nabla\rho\ek \right)\nabla \pi\ek\big\|_{L^p}\\
  &\hspace{1.5cm}+\|\sigma\kk \Delta \pi\ek \|_{L^p}\bigg)\\
  &\leq  Mm^{-4}\bigg(M\|\nabla \sigma\kk \nabla \pi\ek\|_{L^p} + \| \sigma\kk \nabla\rho\kk \nabla \pi\ek\|_{L^p}\\
  &\hspace{1.5cm} +\|\sigma\kk \nabla\rho\ek \nabla \pi\ek\|_{L^p}\bigg)+ m^{-2}\|\sigma\kk\|_{L^p}\|\nabla \pi\ek\|_{2,p}\\
&\leq  Mm^{-4}\left(M L_2 \| \nabla \sigma\kk\|_{L^p}+2 L_1 L_2\|\sigma\kk \|_{L^p}+L_2\| \sigma\kk\|_{1,p}\right)+ L_2m^{-2}\|\sigma\kk\|_{L^p}.
\end{align*}
The computations behind the estimates for the norm of the terms $I_2$ and $I_3$ are similar, so we will only show the first one:
\begin{align*}
\|I_2\|_{L^p}&\leq \z^{-2} \sum_{i,j=1}^3\supx |\Tilde{v}_{x_i}\ek| \|\nabla \Tilde{w}\ek\|_{L^p}\leq c_5 \z^{-2} A \|\Tilde{w}\ek\|_{1,p}.
\end{align*}
 Finally, our result follows by applying these estimates on Eq. \eqref{esI}.
\end{proof}

\begin{lemma}\label{hest}
Let $h\kk$ be the solution of the system \eqref{defh}, then 
\begin{equation*}
    \|h\kk(t) \|_{1,p}\leq L_9 \intt \left(\z(s) \| \nabla q\kk(s)\|_{1,p}+\ie(s)  \|\Tilde{w}\ek(s) \|_{1,p}+\z(s)\,\|\sigma\kk (s)\|_{1,p}\right)\, ds,
\end{equation*}
for all $t\in [0, \tau]$.
\end{lemma}
\begin{proof}
Considering $h^{(k)}=(h^{k,1}, h^{k,2}, h^{k,3})$ and $u^{(k)}=(u^{k,1}, u^{k,2}, u^{k,3})$ we can rewrite \eqref{defh}, for $i=1, 2, 3$, componentwise by
\begin{align}\label{L3-eq1}
\left\{
\begin{aligned}
& h_t^{k,i} + \ie \Tilde{v}^{(k-1)}\cdot \nabla h^{k,i} + \ie \Tilde{w}\ek\cdot\nabla u^{k-1, i}+\z\frac{ q^{(k)}_{x_i}}{\rho^{(k)}}=\z\frac{\sigma\kk \pi^{(k-1)}_{x_i}}{\rho\ek \rho\kk }  , \\
& h^{k,i}|_{t=0} =0.
\end{aligned}
\right.
\end{align}
Following the same scheme as in the proof of Proposition \ref{proprho} we derive an a priori estimate to $\|h\kk(t) \|^p_{1,p}$. In fact, we realized this in three steps:
\\
\noindent{\it \underline{Step 1}: Estimate of $\displaystyle \frac{1}{p}\frac{d}{dt}\|h^{(k)}\|_{L^p}^p$.} From \eqref{L3-eq1}, for $i=1$, we have
\begin{align}\label{L3-eq2}
\frac{1}{p}\frac{d}{dt}\|h^{k,1}\|_{L^p}^p & =-\ie\int_{\mathbb{R}^3} |h^{k,1}|^{p-1}\text{ sgn}(h^{k,1}) \Tilde{v}^{(k-1)}\cdot \nabla h^{k,1} \ dx \nonumber\\
& \quad -\ie\int_{\mathbb{R}^3} |h^{k,1}|^{p-1}\text{ sgn}(h^{k,1}) \ \Tilde{w}^{(k-1)}\cdot \nabla u^{k-1,1} \ dx\nonumber\\
& \quad -\z\int_{\mathbb{R}^3} |h^{k,1}|^{p-1}\text{ sgn}(h^{k,1}) \ \frac{ q^{(k)}_{x_i}}{\rho^{(k)}} \ dx\nonumber\\
& \quad +\z\int_{\mathbb{R}^3} |h^{k,1}|^{p-1}\text{ sgn}(h^{k,1}) \ \frac{\sigma\kk \pi^{(k-1)}_{x_i}}{\rho\ek \rho\kk } \ dx\nonumber\\
& = A_1+A_2+A_3+A_4\,,
\end{align}
where $A_1=0$ since $div \ \Tilde{v}^{(k-1)}=0$. Also, note that
\begin{align*}
A_2 & \leq \ie\int_{\mathbb{R}^3} |h^{k,1}|^{p-1} \ |\Tilde{w}^{(k-1)}| \ |\nabla u^{k-1,1}| \ dx \\
& \leq \ie \sup_{x\in\mathbb{R}^3}\{|\nabla u^{k-1,1}|\}\int_{\mathbb{R}^3} |h^{k,1}|^{p-1} \ |\Tilde{w}^{(k-1)}|  \ dx \\
& \leq \ie \|u^{(k-1)}\|_{2,p} \|h^{k,1}\|^{p-1}_{L^p} \ \|\Tilde{w}^{(k-1)}\|_{L^p}   \\
& \leq \ie A \|h^{(k)}\|^{p-1}_{L^p} \ \|\Tilde{w}^{(k-1)}\|_{L^p} \,.
\end{align*}
In a similar way it results
\begin{align*}
A_3 & \leq \z\int_{\mathbb{R}^3} |h^{k,1}|^{p-1} \ \frac{ |q^{(k)}_{x_1}|}{|\rho^{(k)}|} \ dx \leq \z m^{-1}\int_{\mathbb{R}^3} |h^{k,1}|^{p-1} \ |q^{(k)}_{x_1}| \ dx\\
& \leq \z m^{-1} \|h^{(k)}\|^{p-1}_{L^p} \ \|\nabla q^{(k)}\|_{L^p}\,,
\end{align*}
and
\begin{align*}
A_4 & \leq \z\int_{\mathbb{R}^3} |h^{k,1}|^{p-1} \ \frac{|\sigma\kk| \ |\pi^{(k-1)}_{x_1}|}{|\rho\ek \rho\kk| } \ dx \leq \z m^{-2}\int_{\mathbb{R}^3} |h^{k,1}|^{p-1} \ |\sigma\kk| \ |\pi^{(k-1)}_{x_1}| \ dx\\
& \leq \z m^{-2}\sup_{x\in\mathbb{R}^3}\{|\pi^{(k-1)}_{x_1}|\}\int_{\mathbb{R}^3} |h^{k,1}|^{p-1} \ |\sigma\kk|  \ dx \leq \z m^{-2} \|\nabla \pi^{(k-1)}\|_{2,p} \ \|h^{k,1}\|^{p-1}_{L^p} \ \|\sigma\kk \|_{L^p} \\
& \leq \z m^{-2} L_2 \ \|h^{(k)}\|^{p-1}_{L^p} \ \|\sigma\kk \|_{L^p}\,.
\end{align*}
So, by replacing the estimates of $A_1, A_2, A_3$ and $A_4$ in \eqref{L3-eq2}, we find
\begin{align*}
\frac{1}{p}\frac{d}{dt}\|h^{k,1}\|_{L^p}^p & \leq \ie A \|h^{(k)}\|^{p-1}_{L^p} \ \|\Tilde{w}^{(k-1)}\|_{L^p} \\
& \quad +\z m^{-1} \|h^{(k)}\|^{p-1}_{L^p} \ \|\nabla q^{(k)}\|_{L^p}\\
& \quad + \z m^{-2} L_2 \ \|h^{(k)}\|^{p-1}_{L^p} \ \|\sigma\kk \|_{L^p}\,.
\end{align*}
From this, one may easily verify that the same estimate also holds for $h^{k,2}$ and $h^{k,3}$. Therefore,
\begin{align}\label{L5-eq3}
\frac{1}{p}\frac{d}{dt}\|h^{(k)}\|_{L^p}^p & \leq 3\ie A \|h^{(k)}\|^{p-1}_{L^p} \ \|\Tilde{w}^{(k-1)}\|_{L^p} \nonumber\\
& \quad +3 \z m^{-1} \|h^{(k)}\|^{p-1}_{L^p} \ \|\nabla q^{(k)}\|_{L^p}\nonumber\\
& \quad + 3 \z m^{-2} L_2 \ \|h^{(k)}\|^{p-1}_{L^p} \ \|\sigma\kk \|_{L^p}\,.
\end{align}
\\
\noindent{\it \underline{Step 2}: Estimate of $\displaystyle \frac{1}{p}\frac{d}{dt}\|\nabla h^{(k)}\|_{L^p}^p$.} One more time, from \eqref{L3-eq1}, for $i=1$, by considering the corresponding first derivative of each component of $h^{(k)}$ we deduce:
\begin{align}\label{L6-eq0}
\frac{1}{p}\frac{d}{dt}\|h^{k,1}_{x_j}\|_{L^p}^p & = -\ie\int_{\mathbb{R}^3} |h^{k,1}_{x_j}|^{p-1} \text{ sgn}(h^{k,1}_{x_j}) \ \Tilde{v}^{(k-1)}_{x_j}\cdot \nabla h^{k,1} \ dx \nonumber\\
& \quad -\ie\int_{\mathbb{R}^3} |h^{k,1}_{x_j}|^{p-1} \text{ sgn}(h^{k,1}_{x_j}) \ \Tilde{v}^{(k-1)}\cdot \nabla h^{k,1}_{x_j} \ dx \nonumber\\
& \quad -\ie\int_{\mathbb{R}^3} |h^{k,1}_{x_j}|^{p-1} \text{ sgn}(h^{k,1}_{x_j}) \ \Tilde{w}^{(k-1)}_{x_j}\cdot \nabla u^{k-1,1} \ dx \nonumber\\
& \quad -\ie\int_{\mathbb{R}^3} |h^{k,1}_{x_j}|^{p-1} \text{ sgn}(h^{k,1}_{x_j}) \ \Tilde{w}^{(k-1)}\cdot \nabla u^{k,1}_{x_j} \ dx \nonumber\\
& \quad +\z\int_{\mathbb{R}^3} |h^{k,1}_{x_j}|^{p-1} \text{ sgn}(h^{k,1}_{x_j}) \ \frac{\rho^{(k)}_{x_j} \ q^{(k)}_{x_j}}{(\rho^{(k)})^2} \ dx \nonumber\\
& \quad -\z\int_{\mathbb{R}^3} |h^{k,1}_{x_j}|^{p-1} \text{ sgn}(h^{k,1}_{x_j}) \ \frac{ q^{(k)}_{x_1 x_j}}{\rho^{(k)}} \ dx \nonumber\\
& \quad +\z\int_{\mathbb{R}^3} |h^{k,1}_{x_j}|^{p-1} \text{ sgn}(h^{k,1}_{x_j}) \ \frac{ \sigma^{(k)} \ \pi^{(k-1)}_{x_1 x_j}}{\rho^{(k)}\rho^{(k-1)}} \ dx \nonumber\\
& \quad +\z\int_{\mathbb{R}^3} |h^{k,1}_{x_j}|^{p-1} \text{ sgn}(h^{k,1}_{x_j}) \ \frac{ \sigma^{(k)}_{x_j} \ \pi^{(k-1)}_{x_1}}{\rho^{(k)}\rho^{(k-1)}} \ dx \nonumber\\
&\quad -\z\int_{\mathbb{R}^3} |h^{k,1}_{x_j}|^{p-1} \text{ sgn}(h^{k,1}_{x_j}) \ \left(\frac{ \sigma^{(k)} \ \pi^{(k-1)}_{x_1}[\rho^{(k)}_{x_j}\rho^{(k-1)}+\rho^{(k)}\rho^{(k-1)}_{x_j}]}{(\rho^{(k)}\rho^{(k-1)})^2}\right) \ dx \nonumber\\
& = B_1+B_2+B_3+B_4+B_5+B_6+B_7+B_8+B_9\,,
\end{align}
where $j=1, 2, 3$. Observe that due to $div \  \Tilde{v}^{(k-1)}=0$ we have $B_2=0$. For the other terms of \eqref{L6-eq0}, using the bounds of the Corollary \ref{coroest}, it is not difficult to find out that:
\begin{align*}
B_1 & \leq \ie \int_{\mathbb{R}^3} |h^{k,1}_{x_j}|^{p-1} \ |\Tilde{v}^{(k-1)}| \ |\nabla h^{k,1}| \ dx \leq \ie \ \|\Tilde{v}^{(k-1)}\|_{2,p} \|h^{k,1}_{x_j}\|^{p-1}_{L^p}  \ \|\nabla h^{k,1}\|_{L^p} \\
& \leq \ie \ A \ \|h^{(k)}\|^{p}_{1,p}\,,
\end{align*}
\begin{align*}
B_3 & \leq \ie\int_{\mathbb{R}^3} |h^{k,1}_{x_j}|^{p-1} \ |\Tilde{w}^{(k-1)}_{x_j}| \ |\nabla u^{k-1,1}| \ dx \leq \ie \ A \ \|h^{(k)}\|^{p-1}_{1,p} \ \|\Tilde{w}^{(k-1)}\|_{1,p}\,,
\end{align*}
\begin{align*}
B_4 & \leq \ie\int_{\mathbb{R}^3} |h^{k,1}_{x_j}|^{p-1} \ |\Tilde{w}^{(k-1)}| \ |\nabla u^{k,1}_{x_j}| \ dx \leq \ie \ A \ \|h^{(k)}\|^{p-1}_{1,p} \ \|\Tilde{w}^{(k-1)}\|_{1,p}\,,
\end{align*}
\begin{align*}
B_5 & \leq \z\int_{\mathbb{R}^3} |h^{k,1}_{x_j}|^{p-1} \ \frac{|\rho^{(k)}_{x_j}| \ |q^{(k)}_{x_j}|}{(\rho^{(k)})^2} \ dx \leq \z m^{-2}  \ L_1 \ \|h^{(k)}\|^{p-1}_{1,p} \ \|\nabla q^{(k)}\|_{1,p}\,,
\end{align*}
\begin{align*}
B_6 & \leq \z\int_{\mathbb{R}^3} |h^{k,1}_{x_j}|^{p-1} \ \frac{ |q^{(k)}_{x_1 x_j}|}{|\rho^{(k)}|} \ dx \leq \z m^{-1} \ \|h^{(k)}\|^{p-1}_{1,p} \ \|\nabla q^{(k)}\|_{1,p}\,,
\end{align*}
\begin{align*}
B_7 & \leq \z\int_{\mathbb{R}^3} |h^{k,1}_{x_j}|^{p-1} \ \frac{ |\sigma^{(k)}| \ |\pi^{(k-1)}_{x_1 x_j}|}{|\rho^{(k)}\rho^{(k-1)}|} \ dx \leq \z m^{-2} \ L_2 \ \|h^{(k)}\|^{p-1}_{1,p} \  \|\sigma^{(k)}\|_{1,p}\,,
\end{align*}
\begin{align*}
B_8 & \leq \z\int_{\mathbb{R}^3} |h^{k,1}_{x_j}|^{p-1} \ \frac{ |\sigma^{(k)}_{x_j}| \ |\pi^{(k-1)}_{x_1}|}{|\rho^{(k)}\rho^{(k-1)}|} \ dx \leq \z m^{-2} \ L_2 \ \|h^{(k)}\|^{p-1}_{1,p} \  \|\sigma^{(k)}\|_{1,p}\,,
\end{align*}
and finally
\begin{align*}
B_9 & \leq \z\int_{\mathbb{R}^3} |h^{k,1}_{x_j}|^{p-1} \ \left(\frac{ |\sigma^{(k)}| \ |\pi^{(k-1)}_{x_1}|\left[|\rho^{(k)}_{x_j}||\rho^{(k-1)}|+|\rho^{(k)}||\rho^{(k-1)}_{x_j}|\right]}{(\rho^{(k)}\rho^{(k-1)})^2}\right) \ dx\\
& \leq \z m^{-4} \ M \ \sup_{x\in\mathbb{R}^3}\{|\rho^{(k)}_{x_j}|+|\rho^{(k-1)}_{x_j}|\} \ \sup_{x\in\mathbb{R}^3}\{|\pi^{(k-1)}_{x_1}|\} \ \int_{\mathbb{R}^3} |h^{k,1}_{x_j}|^{p-1} \ |\sigma^{(k)}| \ dx\\
& \leq 2\z m^{-4} \ M L_1 L_2 \ \|h^{(k)}\|^{p-1}_{1,p} \  \|\sigma^{(k)}\|_{1,p}\,.
\end{align*}
Then, substituting the previous estimates in \eqref{L6-eq0} we infer
\begin{align}\label{L8-eq1}
\frac{1}{p}\frac{d}{dt}\|h^{k,1}_{x_j}\|_{L^p}^p & \leq \ie A \ \|h^{(k)}\|^{p}_{1,p}\nonumber\\
& \quad + \|h^{(k)}\|^{p-1}_{1,p}\left[ 2\ie \ A \ \|\Tilde{w}^{(k-1)}\|_{1,p}+\z \ \Tilde{L}_7 \ \|\nabla q^{(k)}\|_{1,p} +\z \ \Tilde{L}_8 \  \|\sigma^{(k)}\|_{1,p}\right]\,.
\end{align}
where $\Tilde{L}_7=m^{-2}  \ L_1+m^{-1}$ and $\Tilde{L}_8=2m^{-2} \ L_2+ 2m^{-4} \ M L_1 L_2$. Note that this estimate is also valid for $h^{k,2}_{x_j}$ and $h^{k,3}_{x_j}$. Therefore,
\begin{align}\label{L8-eq3}
\frac{1}{p}\frac{d}{dt}\|\nabla h^{(k)}\|_{L^p}^p & \leq 3\ie A \ \|h^{(k)}\|^{p}_{1,p}\nonumber\\
& \quad + 3\|h^{(k)}\|^{p-1}_{1,p}\left[ 2\ie A \ \|\Tilde{w}^{(k-1)}\|_{1,p}+\z \ \Tilde{L}_7 \ \|\nabla q^{(k)}\|_{1,p} +\z \ \Tilde{L}_8 \  \|\sigma^{(k)}\|_{1,p}\right]\,.
\end{align}
\\
\noindent{\it \underline{Step 3}: Conclusion.}  Combining \eqref{L5-eq3} and \eqref{L8-eq3}, we conclude that
\begin{align*}
\frac{1}{p}\frac{d}{dt}\|h^{(k)}\|_{1,p}^p & \leq c_6\left[\ie A \ \|h^{(k)}\|^{p}_{1,p}\right. \nonumber\\
& \quad \left.+\|h^{(k)}\|^{p-1}_{1,p}\left(\ie A \ \|\Tilde{w}^{(k-1)}\|_{1,p}+\z \ L_7 \ \|\nabla q^{(k)}\|_{1,p} +\z \ L_8 \  \|\sigma^{(k)}\|_{1,p}\right)\right]\,,
\end{align*}
with $L_7$ and $L_8$ defined by
\begin{align*}
L_7 & = m^{-1}+\Tilde{L}_7\\
L_8 & = m^{-2} L_2+\Tilde{L}_8\,.
\end{align*}
Hence,
\begin{align*}
\frac{d}{dt}\|h^{(k)}\|_{1,p} & \leq c_6\left[\ie A \ \|h^{(k)}\|_{1,p}\right. \nonumber\\
& \quad \left. + \ie A \ \|\Tilde{w}^{(k-1)}\|_{1,p}+\z \ L_7 \ \|\nabla q^{(k)}\|_{1,p} +\z \ L_8 \  \|\sigma^{(k)}\|_{1,p}\right]\,.
\end{align*}
Again, applying Gronwall's inequality we obtain the claim.

\end{proof}

At this point, we can finally prove the convergence of the sequence $\{\Tilde{v}^{(k)}\}_k$, which will be the key to proving the convergence of the other sequence of approximations.

\begin{proposition}\label{vk-limit}
    The sequence $\{\Tilde{v}^{(k)}\}_k$ converge in $C([0,\tau];W^{1,p}(\RR^3))$ to a random variable $\tilde{v}\in C([0, \tau]; W^{2,p}(\RR^3))$.
\end{proposition}
\begin{proof}
 We have deduce from \eqref{eqphi}
\begin{align*}
    \| \Tilde{w}\kk \|_{1,p}&\leq \| h\kk \|_{1,p}+ \| \nabla (\phi\kk-\phi\ek) \|_{1,p}\leq c_9 \| h\kk \|_{1,p},
\end{align*}
using the previous lemmas 
\begin{align*}
    \|\Tilde{w}\kk(t)\|_{1,p}\leq L_{10} \intt \ie(s) \| \Tilde{w}\ek(s) \|_{1,p} \,ds,
\end{align*}
from which it follows that
\begin{align*}
    \|\Tilde{w}\kk(t)\|_{1,p}&\leq L_{10}^{k-1} \frac{t^{k-1}}{(k-1)!}\sup_{0\leq s\leq t}(\ie(s))^{k-1} \sup_{0\leq s\leq t}\| \Tilde{w}^1(s) \|_{1,p},
\end{align*}
then 
\begin{align*}
    \sup_{0\leq t\leq \tau}\|\Tilde{w}\kk(t)\|_{1,p}
    &\leq A\,L_{10}^{k-1} \sup_{0\leq t\leq \tau}(\ie(s))^{k-1} \frac{\tau^{k-1}}{(k-1)!} ,
\end{align*}
therefore
\begin{align}\label{wk-series}
\sum_{k=1}^\infty\|\Tilde{w}\kk\|_{C([0,\tau];W^{1,p}(\mathbb{R}^3))}<\infty.
\end{align}  
This implies $\{\Tilde{v}^{(k)}\}_k$ is a Cauchy sequence and converges to an element $\Bar{v}$ in $C([0,\tau];W^{1,p}(\RR^3))$. By Proposition \ref{propbouv} this sequence is also bounded in $C([0, \tau]; W^{2,p}(\RR^3))$, i.e. $\Tilde{v}^{(k)}$ belong to $B^2(A)$, the ball in $C([0, \tau]; W^{2,p})$ of radius $A$. Then, by the Banach-Alaoglu theorem, there exists a subsequence $\{\Tilde{v}^{(k_j)}\}_j$ that weakly converges to an element $\Tilde{v} \in C([0, \tau]; W^{2,p}(\RR^3))$. Since $B^2(A)$ is a convex set, by Mazur's theorem this limit $\Tilde{v}$ also belongs to this ball.  Finally, we can deduce that $\Bar{v} = \Tilde{v}$ as elements in $C([0, \tau]; W^{1,p}(\RR^3))$  since they are the weak limit in this space of the subsequence considered.
\end{proof}

Now, we present the convergence results for the remaining sequences of approximations and prove that the corresponding limits are the desired solutions to our problems.

\begin{lemma}\label{limit-rhok}
    The sequence $\{\rho^{(k)}\}_k$ converges in $C([0,\tau]\times \mathbb{R}^3)$ to a random variable $\rho\in C([0, \tau]\times\mathbb{R}^3)$.
\end{lemma}

\begin{proof}
First, remember that $\sigma\kk=\rho^{(k)}-\rho^{(k-1)}$. Now, by Lemma \ref{sigmaest}, we have for $t\in[0,\tau]$
\begin{align}\label{estrhow}
\|\sigma\kk(t)\|_{1,p} & \leq L_5 \intt \ie(s) \| \Tilde{w}\ek(s)\|_{1,p}\,ds \leq C \sup_{t\in[0,\tau]} \| \Tilde{w}\ek(t)\|_{1,p}\\ \nonumber
& \leq C \ \| \Tilde{w}^{(k-1)}\|_{C([0,\tau]; W^{1,p}(\mathbb{R}^3))}\,,
\end{align}
that is,
\begin{align*}
\|\sigma\kk\|_{C([0,\tau]; W^{1,p}(\mathbb{R}^3))}\leq C \ \| \Tilde{w}^{(k-1)}\|_{C([0,\tau]; W^{1,p}(\mathbb{R}^3))}\,.
\end{align*}
And by Eq. \eqref{wk-series} it is clear that
\begin{align}\label{C2-eq1}
\lim_{k\to \infty} \|\sigma\kk\|_{C([0,\tau]; W^{1,p}(\mathbb{R}^3))}=0\,.
\end{align}
To follow, observe that, for each $t\in[0,\tau]$, $\sigma^{(k)}(t)\in W^{1,p}(\mathbb{R}^3)$, with $p>3$. Then, Morrey's inequality implies 
$\sigma^{(k)}(t)\in C^{0,1-\frac{3}{p}}(\mathbb{R}^3)$ and 
\begin{align*}
\|\sigma\kk(t)\|_{C^{0,1-\frac{3}{p}}(\mathbb{R}^3)}\leq C_1\|\sigma\kk(t)\|_{1,p}\,,
\end{align*}
where $C_1$ is a positive constant depending on $p$. Consequently,
\begin{align*}
\|\sigma\kk\|_{C([0,\tau]\times\mathbb{R}^3)}\leq C_1\| \sigma^{(k)}\|_{C([0,\tau]; W^{1,p}(\mathbb{R}^3))}\,.
\end{align*}
From the limit \eqref{C2-eq1}, we get
\begin{align*}
\lim_{k\to \infty} \|\sigma\kk\|_{C([0,\tau]\times\mathbb{R}^3)}=0\,.
\end{align*}
Hence, the sequence $\{\rho^{(k)}\}_k$ is a cauchy sequence and converges to a process $\rho \in C([0,\tau]\times\mathbb{R}^3)$.
\end{proof}

\begin{proposition}\label{solution-rho}
    The random variable $\rho\in C([0, \tau]\times\mathbb{R}^3)$ with $\nabla \rho \in C([0, \tau], W^{1,p}(\mathbb{R}^3))$ found in the Lemma \ref{limit-rhok} is a weak solution of the equation
    \begin{align}\label{equation-rho}
    \left\{
    \begin{aligned}
       & \rho_t +  \ie  \Tilde{v}\cdot \nabla \rho =0\\
       & \rho|_{t=0}=\rho_0(x)
       \end{aligned}
       \right.\,,
 \end{align}
 where $\Tilde{v}$ is the random variable from the Proposition \ref{vk-limit}.

\end{proposition}

\begin{proof}
We know that the process $\rho^{(k)}$ is a weak solution of the Eq. \eqref{iterho} if, for all $\varphi\in C^{\infty}_c (\mathbb{R}^3)$, it verifies the integral equation given by
\begin{align*}
\int_{\mathbb{R}^3} & \rho^{(k)}(t\wedge \tau,x)\varphi(x) \ dx=\int_{\mathbb{R}^3}\rho_0(x) \varphi(x)  dx\\
& \quad \quad \quad\quad \quad \quad +\int_0^{t\wedge \tau}\int_{\mathbb{R}^3} \ie(s) \rho^{(k)}(s,x)\Tilde{v}^{(k-1)}(s,x)\cdot \nabla \varphi(x)\, dx \,ds\,.
\end{align*}
Thereby, taking the limit in the above equation, we obtain
\begin{align}\label{C3-eq1}
& \left|\int_{\mathbb{R}^3} \rho^{(k)}(t\wedge \tau,x)\varphi(x) \ dx  -\int_{\mathbb{R}^3}  \rho(t\wedge \tau,x)\varphi(x) \ dx\right|\nonumber\\
& \quad\quad\quad\quad\quad \leq \int_{\mathbb{R}^3} |\rho^{(k)}(t\wedge \tau,x)-\rho(t\wedge \tau,x)| \  |\varphi(x)| \ dx \nonumber\\
& \quad\quad\quad\quad\quad \leq \|\rho^{(k)}-\rho\|_{C([0,\tau]\times\mathbb{R}^3)} \|\varphi\|_{L^{1}(\mathbb{R}^3)}  \longrightarrow 0\,,\quad \text{ as } \quad k \to \infty\,.
\end{align}
In the same way, notice that
\begin{align*}
& \left|\int_0^{t\wedge \tau}\int_{\mathbb{R}^3} \ie(s) \rho^{(k)}(s,x)\Tilde{v}^{(k-1)}(s,x)\cdot \nabla \varphi(x)\, dx \,ds-\int_0^{t\wedge \tau}\int_{\mathbb{R}^3} \ie(s) \rho(s,x)\Tilde{v}(s,x)\cdot \nabla \varphi(x)\, dx \,ds\right|\\
& \quad\quad\quad\quad\quad \leq \int_0^{t\wedge \tau}\int_{\mathbb{R}^3} \ie(s) |\rho^{(k)}(s,x)-\rho(s,x)| \ |\Tilde{v}^{(k-1)}(s,x)| \ |\nabla \varphi(x)|\, dx \,ds\\
& \quad\quad\quad\quad\quad + \int_0^{t\wedge \tau}\int_{\mathbb{R}^3} \ie(s) |\rho(s,x)| \ |\Tilde{v}^{(k-1)}(s,x)-\Tilde{v}(s,x)| \ |\nabla \varphi(x)|\, dx \,ds\\
& \quad\quad\quad\quad\quad = J_1+J_2\,.
\end{align*}
Thus, by H\"{o}lder's inequality and considering $\frac{1}{p}+\frac{1}{q}=1$, we have
\begin{align}\label{C3-eq2}
J_1 & \leq \|\rho^{(k)}-\rho\|_{C([0,\tau]\times\mathbb{R}^3)} \int_0^{t\wedge \tau}\int_{\mathbb{R}^3} \ie(s) \ |\Tilde{v}^{(k-1)}(s,x)| \ |\nabla \varphi(x)|\, dx \,ds\nonumber\\
& \leq \|\rho^{(k)}-\rho\|_{C([0,\tau]\times\mathbb{R}^3)} \sup_{s\in[0,t\wedge \tau]}\{\ie(s)\} \int_0^{t\wedge \tau} \|\Tilde{v}^{(k-1)}(s)\|_{L^p(\mathbb{R}^3)} \ \|\nabla \varphi\|_{L^q(\mathbb{R}^3)}\,ds \nonumber\\
& \leq T \ \|\rho^{(k)}-\rho\|_{C([0,\tau]\times\mathbb{R}^3)} \sup_{s\in[0, \tau]}\{\ie(s)\}  \sup_{s\in[0,\tau]} \{\|\Tilde{v}^{(k-1)}(t)\|_{2,p}\} \ \|\nabla \varphi\|_{L^q(\mathbb{R}^3)}\nonumber\\
& \leq TA \ \|\rho^{(k)}-\rho\|_{C([0,\tau]\times\mathbb{R}^3)} \sup_{s\in[0,\tau]}\{\ie(s)\}  \ \|\nabla \varphi\|_{L^q(\mathbb{R}^3)} \longrightarrow 0\,,\quad \text{ as } \quad k \to \infty\,.
\end{align}
Also, it is not difficult to see that
\begin{align}\label{C3-eq3}
J_2 & \leq M T \sup_{s\in[0, \tau]}\{\ie(s)\} \ \|\nabla \varphi\|_{L^q(\mathbb{R}^3)} \|\Tilde{v}^{(k-1)}-\Tilde{v}\|_{C([0,\tau]; W^{1,p}(\mathbb{R}^3))} \nonumber\\
& \quad\quad\longrightarrow 0\,,\quad \text{ as } \quad k \to \infty\,.
\end{align}
From \eqref{C3-eq1}, \eqref{C3-eq2} and \eqref{C3-eq3}, we conclude that $\rho$ is a weak solution of the Eq. \eqref{equation-rho}.
\end{proof}

We continue now dealing with the convergence of the sequence $\{\nabla \pi^{(k)}\}_k$.

\begin{proposition}\label{limit-pik}
The sequence $\{\nabla \pi^{(k)}\}_k$ converges in $C([0,\tau]; W^{1,p}(\RR^3))$ to a random variable $\nabla \pi \in C([0,\tau]; W^{2,p}(\RR^3))$.  Moreover, $\nabla \pi$ is a weak solution to the problem
 \begin{align}\label{div}
     \text{div}( \rho^{-1}(t)   \nabla \pi(t))=-\z^{-2}(t)\text{div}(\Tilde{v}(t)\cdot \nabla \Tilde{v}(t)),
 \end{align}
 for every $t\in [0,\tau]$.
\end{proposition}
\begin{proof}
First, remember $q\kk=\pi\kk-\pi\ek$. By Lemma \ref{qest} and Eq. \eqref{estrhow} we obtain for $t\in [0, \tau]$ that
\begin{align*}
    \| \nabla q\kk(t)\|_{1,p}&\leq L_6\left(   \|\sigma\kk(t)\|_{1,p}+\z^{-2}(t)\|\Tilde{w}\ek(t)\|_{1,p}\right)\\
    &\leq L_6 \ \left( C+ \sup_{t\in[0,\tau]}\{\z^{-2}(t)\}\right) \|\Tilde{w}^{(k-1)}\|_{C([0,\tau]; W^{1,p}(\mathbb{R}^3))} \\
    &=C_1 \ \|\Tilde{w}^{(k-1)}\|_{C([0,\tau]; W^{1,p}(\mathbb{R}^3))}.
\end{align*}
In particular,
\begin{align*}
    \| \nabla q\kk\|_{C([0,\tau]; W^{1,p}(\mathbb{R}^3))}\leq C_1 \ \|\Tilde{w}^{(k-1)}\|_{C([0,\tau]; W^{1,p}(\mathbb{R}^3))}.
\end{align*}
This inequality combined with the result of Eq. \eqref{wk-series} implies that $\{\nabla \pi^{(k)}\}_k$ is a Cauchy sequence in $C([0,\tau]; W^{1,p}(\RR^3))$. Then, there exists a limit $\nabla \pi \in C([0,\tau]; W^{1,p}(\RR^3))$. Since the sequence is also bounded in $C([0,\tau]; W^{2,p}(\RR^3))$ (see Eq. \eqref{esp1}), with the same argument as in the proof of Proposition \ref{vk-limit}, we can conclude that $\nabla \pi \in C([0,\tau]; W^{2,p}(\RR^3)) $.

\medskip

Now, it only remains to prove the limit $\nabla \pi$ is a solution as desired.
First, recall the process $\nabla \pi\kk$ is a weak solution to Problem \ref{itep}. That is, it verifies 
\begin{equation}\label{weaksolution-pk}
    \intr \frac{\nabla \pi\kk(x)}{\rho\kk(x)}\cdot \nabla \varphi(x)\,dx=-\z^{-2}\intr \text{div}\left(\Tilde{v}\kk(x)\cdot \nabla \Tilde{v}\kk(x)  \right) \varphi(x)\,dx,
\end{equation}
for all $\varphi\in C^{\infty}_c (\mathbb{R}^3)$ and  $t\in [0,\tau]$, which is omitted in our notation just for simplicity.
Consequently, we need to show that the corresponding limits also fulfill this equation.
First, conveniently using Lemma \ref{limit-rhok} and H\"older's inequality, the term on the left side converges as wanted, that is
\begin{align*}
   &\left| \intr  \frac{\nabla \pi\kk(x)}{\rho\kk(x)} \cdot \nabla \varphi(x)\,dx-  \intr \frac{\nabla \pi(x)}{\rho(x)}\cdot \nabla \varphi(x)\,dx\right|\\  &\quad\quad\leq\intr \left|\frac{\nabla \pi\kk(x)-\nabla \pi(x)}{\rho\kk(x)}\right|\left|\nabla \varphi (x)\right|\,dx +\intr \left|\frac{\rho\kk(x)-\rho(x)}{\rho\kk(x) \rho(x)}\right|\left|\nabla \pi(x)\right| \left|\nabla \varphi(x)\right|\,dx\\
   &\quad\quad\leq m^{-1}\intr \left|\nabla \pi\kk(x)-\nabla \pi(x)\right|\left|\nabla \varphi (x)\right|\,dx +m^{-2}\supx \left|\rho\kk(x)-\rho(x)\right|\intr \left|\nabla \pi(x)\right| \left|\nabla \varphi(x)\right|\,dx\\
   &\quad\quad\leq\left( m^{-1} \|\nabla \pi\kk-\nabla \pi\|_{L^p(\RR^3)}+ m^{-2}\|\rho^{(k)}-\rho\|_{C(\mathbb{R}^3)}\|\nabla \pi\|_{L^p(\RR^3)}\right)\|\nabla \varphi \|_{L^q(\RR^3)}\\
   & \quad\quad\longrightarrow 0\,,\quad \text{ as } \quad k \to \infty\,.
\end{align*}
Next, we can bound the right side term of Eq.\eqref{weaksolution-pk} as follows
\begin{align*}
   &\left|\z^{-2}\intr \text{div}\left(\Tilde{v}\kk(x)\cdot \nabla \Tilde{v}\kk(x)  \right) \varphi(x)\,dx - \z^{-2}\intr \text{div}\left(\Tilde{v}(x)\cdot \nabla \Tilde{v}(x)  \right)\varphi(x)\,dx\right|\\
   &=\left| \z^{-2}\sum_{i,j=1}^3 \intr \Tilde{v}^{(k-1),i}_{x_j}(x)\Tilde{v}^{(k-1),j}_{x_i}(x) \   \varphi(x)\,dx-\z^{-2}\sum_{i,j=1}^3 \intr \Tilde{v}^{i}_{x_j}(x)\Tilde{v}^{j}_{x_i}(x) \   \varphi(x)\,dx. \right|\\
   &\leq \z^{-2}\sum_{i,j=1}^3 \intr \left|\Tilde{v}^{(k-1),i}_{x_j}(x)\left(\Tilde{v}^{(k-1),j}_{x_i}(x)-\Tilde{v}^{j}_{x_i}(x) \right)\right||\varphi(x)|\,dx\\&+\z^{-2}\sum_{i,j=1}^3 \intr \left|\left(\Tilde{v}^{(k-1),i}_{x_j}(x)-\Tilde{v}^{i}_{x_j}(x) \right)\Tilde{v}^{j}_{x_i}(x)\right||\varphi(x)|\,dx= J_1+J_2.
\end{align*}
Using embedding theorems, Hölder's inequality and Proposition \ref{vk-limit}, we obtain 
\begin{align*}
     J_1 &\leq \z^{-2}  \sum_{i,j=1}^3 \supx |\Tilde{v}^{(k-1),i}_{x_j}| \ \|\Tilde{v}^{(k-1),j}_{x_i}-\Tilde{v}^{j}_{x_i}\|_{L^p(\RR^3)} \|\varphi\|_{L^q(\RR^3)}\\
   &\leq \z^{-2} \, \|\Tilde{v}\ek \|_{W^{2,p}(\RR^3)}\| \Tilde{v}\ek-\Tilde{v} \|_{ W^{1,p}(\RR^3)}\|\varphi\|_{L^q(\RR^3)}\\
   & \quad\quad\longrightarrow 0\,,\quad \text{ as } \quad k \to \infty .
\end{align*}
With a similar estimation, it is easy to show  that $J_2$ also converges to zero. Finally, $\nabla \pi$ satisfies Eq. \eqref{div} for every $t\in [0,\tau]$, as claimed.
\end{proof}

In the following proposition, we will adopt the same proof scheme as in the other convergence results discussed earlier. We will omit several steps, appropriately referencing each with the corresponding arguments previously used.

\begin{proposition}\label{solution-u}
 The sequence $\{ u^{(k)}\}_k$ converges in $C([0,\tau]; W^{1,p}(\RR^3))$ to a random variable $u \in C([0,\tau]; W^{2,p}(\RR^3))$. Moreover, $u$ is a weak solution to the problem
\begin{equation}\label{u}
\begin{cases}  
 u_t + \ie \Tilde{v}\cdot \nabla u+ \z\frac{\nabla \pi}{\rho}=0,\\
u|_{t=0} =v_0(x). &
\end{cases}
\end{equation}   
\end{proposition}

\begin{proof}
By Lemma \ref{hest} we obtain
\begin{align*}
    \|h\kk(t)\|_{1,p}&\leq L_{11} \intt \ie(s) \| \Tilde{w}\ek(s) \|_{1,p} \,ds\\
    &\leq L_{11} \sup_{t\in[0,\tau]}\{\z^{-2}(t)\}\|\Tilde{w}\ek(s) \|_{C([0,\tau]; W^{1,p}(\mathbb{R}^3))}\\
    &\leq L_{12}\|\Tilde{w}\ek(s) \|_{C([0,\tau]; W^{1,p}(\mathbb{R}^3))}.
\end{align*}
This implies the sequence $\{ u^{(k)}\}_k$ is Cauchy, implying it has a limit $u \in C([0,\tau]; W^{1,p}(\mathbb{R}^3))$. Now, using the estimates described at the beginning of the proof of Lemma \ref{propbouv},  we get
\begin{equation*}
    \|u\kk(t)\|_{2,p}\leq  \text{ exp}\left(c_2\right)\cdot\big[\|v_0\|_{2,p}+  L(\|\nabla\rho_0\|_{1,p}\text{ exp}\left(c_1\right))\big],
\end{equation*}
for  all $t\in [0,\tau]$.
Consequently, with the same argument as in the proof of Proposition \ref{vk-limit}, we can also deduce that $u \in C([0,\tau]; W^{2,p}(\RR^3)) $.

To conclude the proof, we need to show $u$ is a weak solution to Problem \ref{u}. First, recall that
the process $u\kk$ is a weak solution to Problem \ref{iteu}. In particular, this means that if we write this process componentwise as $u^{(k)}=(u^{k,1}, u^{k,2}, u^{k,3})$, the following equation is  satisfied:
\begin{align*}
\int_{\mathbb{R}^3} u^{k,i}(t\wedge \tau,x)\varphi(x) \ dx&=\int_{\mathbb{R}^3} u_0^{k,i}(x) \varphi(x)  dx\\&\quad\quad -\int_0^{t \wedge \tau} \intr\ie(s)\Tilde{v}\kk(s,x) u^{k,i}(s,x)\nabla \varphi(x)\,dx\,ds\\&\quad\quad\quad\quad-\int_0^{t \wedge \tau}\intr \z(s) \ \frac{ \pi_{x_i}\kk(s,x)}{\rho\kk(s,x)}\varphi(x)\,dx\,ds, 
\end{align*}
for $i=1,2,3$  and for all $\varphi\in C^{\infty}_c (\mathbb{R}^3)$. 
We want to show the above equation is precisely satisfied by the corresponding limits of each of the sequences involved. First, with a similar argument to the one described in the proof of  Proposition \ref{propu}, it is not hard to prove that all terms involving the sequence $u^{k,i}$ pass to the corresponding limit. 
Consequently, it only remains to prove the convergence of the last term of the right-hand side of the above equation. We begin showing that
\begin{align*}
&\left|\int_0^{t \wedge \tau}\intr \z(s) \ \frac{ \pi_{x_i}\kk(s,x)}{\rho\kk(s,x)}\varphi(x)\,dx\,ds-\int_0^{t \wedge \tau}\intr \z(s) \ \frac{ \pi_{x_i}(s,x)}{\rho(s,x)}\varphi(x)\,dx\,ds\right|\\
&\leq 
\int_0^{t \wedge \tau}\intr \z(s)  \left|\frac{ \pi_{x_i}\kk(s,x)-\pi_{x_i}(s,x)}{\rho\kk(s,x)}\right||\varphi(x)|\,dx\,ds\\
&\quad\quad+\int_0^{t \wedge \tau}\intr \z(s)\, \left|\frac{\rho\kk(s,x)-\rho(s,x)}{\rho\kk(s,x)\rho(s,x)}\right| |\pi_{x_i}(s,x)| |\varphi(x)|\,dx\,ds\\
&= A_1+A_2.
\end{align*}
A straightforward computation combined with the convergence results we already proved, implies that
\begin{align*}
   A_1 &\leq  m^{-1}\int_0^{t \wedge \tau}\z(s)\, \|\pi_{x_i}\kk(s)-\pi_{x_i}(s)\|_{L^p(\RR^3)}\|\varphi\|_{L^q(\RR^3)}\,ds\\
   &\leq T m^{-1} \,\sup_{s\in[0, \tau]}\{\z(s)\} \, \|\nabla \pi\kk-\nabla \pi\|_{C([0,\tau];W^{1,p}(\RR^3)}\|\varphi\|_{L^q(\RR^3)}\,ds\\
   & \quad\quad\longrightarrow 0\,,\quad \text{ as } \quad k \to \infty,
\end{align*}
and
\begin{align*}
    A_2&\leq m^{-2}\|\rho\kk-\rho\|_{C([0,\tau]\times \RR^3)}\int_0^{t \wedge \tau}\z(s)\,  \|\pi_{x_i}(s) \|_{L^p(\RR^3)} \|\varphi\|_{L^q(\RR^3)}\,ds\\
    &\leq T m^{-2}\, \sup_{s\in[0, \tau]}\{\z(s)\} \|\rho\kk-\rho\|_{C([0,\tau]\times \RR^3)} \|\nabla \pi \|_{C([0,\tau];W^{1,p}(\RR^3)} \|\varphi\|_{L^q(\RR^3)}\,ds\\
    & \quad\quad\longrightarrow 0\,,\quad \text{ as } \quad k \to \infty .
\end{align*}
Then, $u$ is a weak solution of Problem \ref{u} and our claim follows.
\end{proof}
The next result will be important to prove that $u=\tilde{v}$.

\begin{lemma}\label{solution-phi}
The sequence $\{ \nabla\phi^{(k)}\}_k$ converges in $C([0,\tau]; W^{1,p}(\RR^3))$ to a random variable $\nabla \phi \in C([0,\tau]; W^{1,p}(\RR^3))$. Besides that, if $\Tilde{v}$ and $u$ are the random variables from Proposition \ref{vk-limit} and Proposition \ref{solution-u}, respectively, then $\nabla \phi$ satisfies
\begin{align}\label{D2-eq1}
\Tilde{v}=u-\nabla \phi,
\end{align}
with $\phi$ being a weak solution of the elliptic equation
\begin{align}\label{D2-eq2}
\Delta \phi = \text{div} \ u.
\end{align}
\end{lemma}

\begin{proof}
First, by \eqref{eqphi}, for $t\in[0,\tau]$, we have
\begin{align*}
\|\nabla \phi^{(k)}(t)-\nabla \phi^{(k-1)}(t)\|_{1,p}=\|\nabla ( \phi^{(k)}(t)- \phi^{(k-1)}(t))\|_{1,p}\leq C\|h^{(k)}(t)\|_{1,p}
\end{align*}
Following the same procedure as in the proof of Proposition \ref{vk-limit}, it is easy to prove that $\{ \nabla\phi^{(k)}\}_k$ is a Cauchy sequence in $C([0,\tau]; W^{1,p}(\RR^3))$. So, there exists a process $\nabla \phi \in C([0,\tau]; W^{1,p}(\RR^3))$ such that
\begin{align}\label{D1-eq1}
\|\nabla \phi^{(k)}-\nabla \phi\|_{C([0,\tau]; W^{1,p}(\RR^3))}\longrightarrow 0 \quad \text{ as } \quad k\to \infty\,.
\end{align}
Passing to the limit as $k\to\infty$ in \eqref{defv}, within the space $C([0,\tau]; W^{1,p}(\RR^3))$, we establish \eqref{D2-eq1}. Now, since $\phi^{(k)}$ is a weak solution of \eqref{eqphi}, then, for all $\varphi\in C^{\infty}_c (\mathbb{R}^3)$, it satisfies 
\begin{align*}
\int_{\mathbb{R}^3} \nabla \phi^{(k)}(x)\cdot\nabla \varphi(x) \ dx = \int_{\mathbb{R}^3} u^{(k)}(x)\cdot\nabla \varphi(x) \ dx
\end{align*}
In order to show that $\phi$ is a weak solution of \eqref{D2-eq2} it is enough to pass to the limit in the above equation. Thereby, by considering the convergence \eqref{D1-eq1} we find
\begin{align*}
& \left|\int_{\mathbb{R}^3}  \nabla \phi^{(k)}(x)\cdot\nabla \varphi(x) \ dx-\int_{\mathbb{R}^3} \nabla \phi(x)\cdot\nabla \varphi(x) \ dx\right|\\
& \quad\quad\quad\|\nabla \phi^{(k)}-\nabla \phi\|_{C([0,\tau]; W^{1,p}(\RR^3))} \|\nabla\varphi\|_{L^q(\RR^3)} \longrightarrow 0\,, \quad \text{ as } \quad k\to 0\,,
\end{align*}
and
\begin{align*}
& \left|\int_{\mathbb{R}^3}  u^{(k)}(x)\cdot\nabla \varphi(x) \ dx-\int_{\mathbb{R}^3} u(x)\cdot\nabla \varphi(x) \ dx\right|\\
& \quad\quad\quad\leq \|u^{(k)}-u\|_{C([0,\tau]; W^{1,p}(\RR^3))} \|\nabla \varphi\|_{L^q(\RR^3)}  \longrightarrow 0\,, \quad \text{ as } \quad k\to 0\,,
\end{align*}
where $1/p+1/q=1$. Hence, the proof is complete.

\end{proof}


Now we are ready to complete the proof of our existence result.

\begin{proposition}[Existence of solutions] \label{existence}
 If $v=\ie \Tilde{v}$, then  ($\rho$, $\nabla \pi$, $v$, $\tau$) determines a local pathwise solution for Problem \ref{eq0}.
\end{proposition}

\begin{proof}
   First, notice that $u=\Tilde{v}$. Indeed, if we apply the divergence operator to each side of Eq. \eqref{u} and use Eq. \eqref{div}, we obtain 
\begin{equation*}
    (\text{div} \ u)_t+ \ie\Tilde{v}\cdot\nabla(\text{div} \ u )=-\ie\sum_{i,j=1}^3\Tilde{v}^i_\xxj\phi_\xij.
\end{equation*}
Applying to this last equation the same estimate as at the beginning of the proof of Proposition \ref{propu} we have 
\begin{equation*}
    \|\text{div} \ u\|_{L^p(\RR^3)} \leq C_2(\tau) \intt  \| \phi\|_{2,p}\,ds.
\end{equation*}
On the other hand, for the elliptic equation \eqref{D2-eq2} the following inequality holds:
\begin{equation*}
    \| \phi \|_{2,p}\leq k \|\text{div} \ u\|_{L^p(\RR^3)} .
\end{equation*}
Then, by Gr\"onwall's inequality and using that $\text{div } u_0=0$, it is easy to see that $\text{div } u=0$. Returning to Eq. \eqref{D2-eq1}, we deduce $\phi=0$.\\

By the Proposition \ref{solution-u}, we have
\begin{equation*}
 \Tilde{v}_t + \ie \Tilde{v}\cdot \nabla \Tilde{v}+ \z\frac{\nabla \pi}{\rho}=0.  
\end{equation*}

Now, by multiplying this equation with the characteristic function of the set $[0, \tau]$ and implementing the variable change $v=\ie \Tilde{v}$ using Stratonovich calculus, we can readily deduce that $v$ satisfies:
 
\[d v + (v\cdot \nabla)v \,dt+ \frac{\nabla \pi}{\rho}\,dt = -v\circ\,d \W_t .
\]
Then our proof is completed by Proposition \ref{solution-rho}.
\end{proof}


\subsection{Uniqueness of solutions}

 The proof of the uniqueness of local solutions follows some of the ideas and estimates we have already used in the preceding subsection.

\begin{proposition}\label{uniqueness}
    The solution to the problem \eqref{eq0} is pathwise unique.
\end{proposition}
\begin{proof}
Let $(\rho^1, \nabla \pi^1, v^1, \tau_1)$ and $(\rho^2, \nabla \pi^2, v^2, \tau_2)$
 be two solutions of the problem with initial value $(\rho_0, v_0)$. We define   $\tau= \tau_1\wedge  \tau_2$, and the variables 
\begin{equation*}
    \Tilde{v}^1(t)=\ie(t) v^1(t)\hspace{0.5cm}\mbox{and} \hspace{0.5cm} \Tilde{v}(t)^2=\ie(t) v^2(t), \hspace{1cm}\text{$t\in [0,\tau ]$.}
\end{equation*} Then, the solutions  satisfy respectively
 \begin{equation*}
    \begin{cases}
        \rho^1_t +  \ie \Tilde{v}^1 \cdot \nabla \rho^1 =0,\\
        \Tilde{v}^1_t+\ie(\Tilde{v}^1\cdot \nabla) \Tilde{v}^1+\frac{1}{\rho^1}\z\, \nabla \pi^1=0,\,\\
       \text{div } \Tilde{v}^1=0,
    \end{cases}
\end{equation*}
and
\begin{equation*}
    \begin{cases}
        \rho^2_t +  \ie\Tilde{v}^2 \cdot \nabla \rho^2 =0,\\
        \Tilde{v}^2_t+\ie(\Tilde{v}^2\cdot \nabla) \Tilde{v}^2+\frac{1}{\rho^2}\z\, \nabla \pi^2=0,\,\\
        \text{div } \Tilde{v}^2=0.
    \end{cases}
\end{equation*}
Subtracting the above equations and considering the differences $\sigma=\rho^2-\rho^1$,  $q=\pi^2-\pi^1 $, $w= v^2- v^1$ and $\Tilde{w}= \Tilde{v}^2- \Tilde{v}^1$, with $\sigma(0)=w(0)=\Tilde{w}(0)=0$.
We obtain

\begin{equation}\label{diffsolutions}
    \begin{cases}
   \sigma_t + \ie \Tilde{v}^2\cdot \nabla \sigma =-\ie \nabla \rho^1 \cdot \Tilde{w},\\   \Tilde{w}_t + \ie \Tilde{v
 }^2\cdot \nabla \Tilde{w}+ \ie \Tilde{w}\nabla \Tilde{v}^1 + \z\frac{\nabla q}{\rho^2} -\z\frac{\sigma}{\rho^2\rho^1}\nabla \pi^1 =0,\\
 \text{div }\Tilde{w}=0.
    \end{cases}
\end{equation}
Now, repeating the same to the second equation of \eqref{diffsolutions} estimates as in the Lemma \ref{hest} taking $h\kk(t)=\Tilde{w}(t)$, we obtain
 \begin{equation*}
     \dt \|\Tilde{w}(t)\|_{1,p}\leq L_{11}\left(\|\Tilde{w}(t) \|_{1,p}+\|\sigma(t) \|_{1,p}+\| \nabla q(t)\|_{1,p} \right),
 \end{equation*}
Applying the divergence operator in the second equation of \eqref{diffsolutions}, we obtain an equation for $q$ as the one in \eqref{defq}. Then, by Lemma \ref{sigmaest} and Lemma \ref{qest} the norms $\|\sigma(t)\|_{1,p}$ and $\|\nabla q(t)\|_{1,p}$ can be bounded by $\|\Tilde{w}(t)\|_{1,p}$, implying
 \begin{equation*}
    \dt \|\Tilde{w}(t)\|_{1,p}\leq L_{12} \|\Tilde{w}(t)\|_{1,p}.
 \end{equation*}
By Gronwall's inequality for each $\omega \in \Omega$  we have $\Tilde{v}^1(t)=\Tilde{v}^2(t)$ with $t\in [0, \tau(\omega)]$.  Finally,  combined with the above estimates let us deduce that
$(\rho^1, \nabla \pi^1, v^1)=(\rho^2, \nabla \pi^2, v^2)$ on $[0,\tau]$, as claimed.
\end{proof}

\medskip

Combining Proposition \ref{existence} and Proposition \ref{uniqueness}, the proof of Theorem \ref{maintheorem} is now complete. 

\subsection{Maximal pathwise solutions}\label{section-maximal}
After proving the well-posedness of local pathwise solutions, it is natural to ask about the existence of a unique maximal pathwise solution. This matter has been discussed in several works using different proof scheme (see, for instance \cite{Bre4}, \cite{Goodair-Crisan-Lang}, \cite{Vicol}). In this article, we will mostly follow the ideas presented in \cite{Goodair-Crisan-Lang}. Broadly speaking, the strategy of the proof relies on a suitable application of Zorn's lemma, for which we will define the following sets.

Let $\mathscr{T}$ be the set of all a.s. strictly positive stopping times corresponding to a local pathwise solution with initial data $(\rho_0, v_0)$. 
Notice that $\mathscr{T}$ is nonempty by the Proposition \ref{existence} and satisfies the following properties:
\begin{itemize}
    \item for $\sigma_1$, $\sigma_2 \in \mathscr{T}$ then $\sigma_1 \vee \sigma_2\in \mathscr{T}$,
    \item for $\sigma_1$, $\sigma_2 \in \mathscr{T}$ then $\sigma_1 \wedge\sigma_2\in \mathscr{T}$.
\end{itemize}
For a detailed proof of the first property, see \cite[Lemma 3.2]{Goodair-Crisan} and \cite[Lemma 3.31]{Goodair-Crisan-Lang}, while the second property is straightforward. 
We also define $\mathscr{F}$ as the set of all the strictly positive stopping times given by a.s. the limit of increasing elements of strictly positive $\mathscr{T}$. This set is nonempty and partially ordered by
\begin{equation*}
  \sigma_1 \leq \sigma_2 \quad \text{if and only if} \quad 
\sigma_1(\omega) \leq \sigma_2(\omega) \ \text{for almost every } \omega .  
\end{equation*}
\begin{proposition}
    There exists a unique maximal pathwise solution to Problem~\eqref{eq0}, in the sense of Definitions~\ref{defmaximalsolution} and~\ref{maxuniqueness}, with initial data $(\rho_0, v_0)$.
\end{proposition}
\begin{proof}
Let $(\Lambda_k)_{k\in\mathbb{N}}\subset\mathscr{F}$ be an increasing sequence, i.e.
\begin{equation*}
\Lambda_1\leq\Lambda_2\leq\cdots\leq\Lambda_n\leq\Lambda_{n+1}\leq\cdots.
\end{equation*}
We are going to show this chain admits an upper bound.
By hypothesis, for each $k\in\mathbb{N}$ there exists an increasing sequence
$(\lambda^k_i)_{i\in\mathbb{N}}\subset\mathscr{T}$ such that $\lambda^k_i\uparrow\Lambda_k$
as $i\to\infty$. For every $n\in\mathbb{N}$ define
\begin{equation*}
 \beta_n:=\bigvee_{k=1}^n \lambda^k_n.   
\end{equation*}
By the first property stated above, namely that $\mathscr{T}$ is closed under finite suprema,
we have $\beta_n\in\mathscr{T}$ for every $n$.
We claim that the sequence $(\beta_n)$ is increasing a.s. Indeed, for $n<m$ and
$1\leq k\leq n$ we have $\lambda^k_n\leq\lambda^k_m$ because $(\lambda^k_i)_i$ is increasing,
hence
\begin{equation*}
 \beta_n=\bigvee_{k=1}^n\lambda^k_n \leq \bigvee_{k=1}^n\lambda^k_m \leq \bigvee_{k=1}^m\lambda^k_m=\beta_m,   
\end{equation*}
so $\beta_n\leq\beta_m$ a.s. Therefore the monotone limit
\begin{equation*}
\Lambda:=\lim_{n\to\infty}\beta_n=\sup_{n\in\mathbb{N}}\beta_n  
\end{equation*}
exists a.s. and, being the increasing limit of stopping times, is itself a stopping time.
By the definition of $\mathscr{F}$, we deduce $\Lambda\in\mathscr{F}$.
Next, observe that for fixed $k$ and all $n\geq k$ we have $\lambda^k_n \leq \beta_n$.
Passing to the limit $n\to\infty$ (using $\lambda^k_n\uparrow\Lambda_k$ and $\beta_n\uparrow\Lambda$)
yields $\Lambda_k \leq \Lambda$. Hence $\Lambda$ is an upper bound of the chain
$\{\Lambda_k:k\in\mathbb{N}\}$ in $\mathscr{F}$.
Since the choice of the increasing sequence $(\Lambda_k)$ was arbitrary, every totally ordered
subset of $\mathscr{F}$ admits an upper bound in $\mathscr{F}$. By Zorn's lemma it follows that 
$\mathscr{F}$ contains at least one maximal element. We denote one such element by $\xi$.
We now construct a solution on $[0,\xi)$. By definition, there exists an increasing sequence $(\sigma_n)_{n\in\mathbb{N}}\subset\mathscr{T}$ converging to $\xi$, such that for each $n$ let
$(\rho_n,\nabla\pi_n,v_n)$ denote the solution corresponding to the stopping time $\sigma_n$.
Because of the uniqueness established in the previous section, these local solutions are
compatible on overlaps: if $m>n$ then the restrictions of $(\rho_m,\nabla\pi_m,v_m)$ and
$(\rho_n,\nabla\pi_n,v_n)$ to $[0,\sigma_n]$ coincide. Therefore the expression
\begin{equation}\label{S4-eq-2}
(\rho,\nabla\pi,v):=(\rho_n,\nabla\pi_n,v_n)\quad\text{on }[0,\sigma_n]
\end{equation}
is well defined and yields a solution on $[0,\sup_n\sigma_n)=[0,\xi)$.
Finally, if $(\rho',\nabla\pi',v', \xi')$ is another maximal solution, then there exists an increasing sequence $(\beta_n)_{n\in\N} \subset \mathscr{T}$ converging a.s. to $\xi'\in \mathscr{F}$, such that each quadruple $(\rho',\nabla\pi',v', \beta_n)$, $n\in \N$, is a local pathwise solution. By the first property of $\mathscr{T}$, the sequence $(\sigma_n \vee \beta_n)_{n\in\N}$ also belongs to $\mathscr{T}$ and converges a.s. to $\xi \vee \xi' \in \mathscr{F}$. Since $\xi, \xi' \leq \xi \vee \xi'$, by the property (ii) in Definition \ref{defmaximalsolution} we conclude that $\xi = \xi \vee \xi' = \xi'$. Moreover, by Proposition~\ref{uniqueness}, which ensures the uniqueness of local pathwise solutions, and by monotone convergence theorem we obtain
\begin{align*}
&\mathbb{P}\Big( (\rho',\nabla\pi', v')(t)
=(\rho,\nabla \pi, v)(t), \ \forall  t\in [0, \xi) \Big)\\&=\mathbb{P}\Big( \bigcup_{n\in \N}\left\{(\rho',\nabla\pi', v')(t)
=(\rho,\nabla \pi, v)(t), \ \forall  t\in [0, \sigma_n\wedge \beta_n]\right\} \Big)\\&=\lim_{n\rightarrow \infty}\mathbb{P}\Big( \bigcup_{k=1}^n \left\{(\rho',\nabla\pi', v')(t)
=(\rho,\nabla \pi, v)(t), \ \forall t\in [0, \sigma_k\wedge \beta_k] \right\}\Big)\\&=\lim_{n\rightarrow \infty}\mathbb{P}\Big( (\rho',\nabla\pi', v')(t)
=(\rho,\nabla \pi, v)(t), \ \forall t\in [0, \sigma_n\wedge \beta_n] \Big)=1.
\end{align*}
\end{proof}

\section{Proof of the theorem \ref{maintheorembi} }

\subsection{Successive approximations}
Let $v^{(0)}=0$, and for $k\in \NN$, let  $\rho^{(k)}$, $\nabla \pi^{(k)}$ and $u^{(k)}$ be the corresponding solutions of the following problems

\begin{equation}\label{iterhobi}
\begin{cases} 
\rho^{(k)}_t +  v^{(k-1)} \cdot \nabla \rho^{(k)} =0, &\\ 
\rho^{(k)}|_{t=0}=\rho_0(x),&
\end{cases}
\end{equation}
\begin{equation}\label{itepbi}
\begin{cases} 
\text{div}\left(\frac{1}{\rho^{(k)}}\nabla \pi^{(k)}\right)=-\sum_{i, j=1}^3  v^{(k-1),i}_{x_j} v^{(k-1),j}_{x_i}, &\\ 
\end{cases}
\end{equation}
\begin{equation}\label{iteubi}
\begin{cases}  
 u_t^{(k)} +  v^{(k-1)}\cdot \nabla u^{(k)}+
   v^{(k-1)}\cdot \nabla \W^{Q}
+  \frac{\nabla \pi^{(k)}}{\rho^{(k)}}=0,\\
u^{(k)}|_{t=0} =v_0(x), &
\end{cases}
\end{equation}
Finally, we define 
\[
\tilde{u}^{(k)}= u^{(k)} + \W^{Q}
\]
\begin{equation}\label{defvbi}
    v^{(k)}= \tilde{u}^{(k)}-\nabla \phi^{(k)}
\end{equation}
where $\phi^{(k)}$ is the solution of 
\begin{equation}\label{eqphibi}
\Delta \phi^{(k)}=\text{div } \tilde{u}^{(k)}.
\end{equation}
The following result is crucial to prove that our successive approximations converge to the desired solution.

\begin{proposition}\label{propbouvbi}
There exists a stopping time $\tau$, such that the sequence $\left\{v^{(k)}\right\}_k$ is bounded in $C([0,\tau];W^{2,p}(\RR^3))$. In particular,
 \begin{equation*}
     \sup_{0\leq t \leq \tau}\|v^{(k)}(t)\|_{2,p}\leq A,
\end{equation*} 
for some constant $A>1$.
\end{proposition}
\begin{proof}
Recalling the estimates of Proposition \ref{propubi} we get
\[
\|u(t)\|_{2,p}\leq \text{exp}\left( c_2\intt \|v
(s)\|_{2,p}  \,ds\right) \times
 \]    
    \[
    \left[\|v_0\|_{2,p}+ \intt  L(\|\nabla \rho(s)\|_{1,p})\| v(s)\|^2_{2,p}  + \| v(s)\|_{2,p}\|\W^{Q}(s)\|_{k,2} \,ds\right].
    \]
Then,
\begin{align*}
    \|u\kk(t)\|_{2,p}&\leq \text{exp}\left( c_2\intt \|v\ek(s)\|_{2,p}      \,ds\right)\\&\hspace{0.5cm}\cdot\left[\|v_0\|_{2,p}+ \intt  L(\|\nabla \rho\kk(s)\|_{1,p})\| v\ek(s)\|^2_{2,p} + \| v\ek(s)\|_{2,p}\|\W^{Q(s)}\|_{k,2}\,ds\right].
\end{align*}
By Eq.\eqref{defvbi} and Eq.\eqref{eqphibi} we have
\begin{equation*}
\|v^{(k)}\|_{2,p}\leq \|\tilde{u}^{(k)}\|_{2,p} + \|\nabla \phi^{(k)}\|_{2,p}\leq c_3 \|\tilde{u}^{(k)}\|_{2,p}
\leq c_3 (\|u^{(k)}\|_{2,p} +   \|\W^{Q}(s)\|_{k,2}),
\end{equation*}
which combined with the above estimate implies
\begin{align*}
    &\quad\quad\quad\quad\quad\quad\|v^{(k)}(t)\|_{2,p} \leq c_3 \text{ exp}\left( c_2\intt  \|v\ek(s)\|_{2,p} \,ds\right)\cdot \\
    &\cdot\left[ \|v_0\|_{2,p}+ \intt \left(L(\|\nabla\rho\kk(s)\|_{1,p})\|v\ek(s)\|^2_{1,p} +\| v\ek(s)\|_{2,p} \|\W^{Q}(s)\|_{k,2} \right)\,ds\right]+ c_{3} \|\W^{Q}(t)\|_{k,2}
\end{align*}
Now, we consider $A>1 $ with 
\begin{align*}
    A /3 \geq c_3 \text{exp}(1)\left[\|v_0\|_{2,p}+L\left(\|\nabla \rho_0\|_{1,p}\text{exp}(1)\right) \right].
\end{align*}
We define 
 \begin{equation*}
        \tau=\inf\{t\geq 0 :
        (c_1\vee c_{2}   \vee 1)A^{2} t + 
      c_{3} \text{exp(1)}  \int_{0}^{t} \|\W^{Q}(s)\|_{k,2} ds + c_3 A^{-1}\|\W^{Q}(t)\|_{k,2}\geq 1/3 \}  \wedge T,
  \end{equation*}
    \ is a stopping time following \cite[Proposition 3.9]{L}.
 Now, we can prove our claim by induction on $k$. The case $k=0$ trivially holds.
 If  we suppose
\begin{equation*}
     \sup_{0\leq t \leq \tau}\|v^{(k-1)}\|_{2,p}\leq A,
\end{equation*}
then using  Proposition \ref{proprho2} , we obtain the following estimate for every $t\leq \tau$:
\begin{align*}
    \|\nabla\rho\kk(t)\|_{1,p}& \leq \|\nabla\rho_0\|_{1,p}\text{ exp}\left(c_1 \intt \|v\ek(s)\|_{2,p} \,ds\right)\\
    &\leq  \|\nabla\rho_0\|_{1,p}\text{ exp}\left(c_1 A t  \right)\\
    &\leq  \|\nabla\rho_0\|_{1,p}\text{ exp}\left(1\right),
\end{align*}
 where the last inequality follows from the definition of the stopping time $\tau$ and the assumption $A\leq A^2$. 
Finally, using the above estimates for $\|v^{(k-1)}\|_{2,p}$ and $\|\nabla\rho\kk\|_{1,p}$, and the fact $L$ are increasing functions, we can deduce 
\begin{align*}
\|v^{(k)}(t)\|_{2,p}&\leq c_3 \text{ exp}\left( c_2 t A  \right)\\
\notag &\hspace{1cm}\cdot\big[\|v_0\|_{2,p}+ L(\|\nabla\rho_0\|_{1,p}\text{ exp}\left(1\right)) t A^2   + \int_{0}^t   \|\W^{Q}(s)\|_{k,2}^{2} \,ds\big]+ c_{3} \|\W^{Q}(t)\|_{k,2}\\
&\leq c_3 \text{ exp}\left(1\right)\cdot\big[\|v_0\|_{2,p}
+L(\|\nabla\rho_0\|_{1,p}\text{exp}\left(1\right))
+ A \int_{0}^t   \|\W^{Q}(s)\|_{k,2}^{2} ds 
\big]+ c_{3} \|\W^{Q}(t)\|_{k,2} \\
& \leq A,
\end{align*}
completing our proof.
\end{proof}
As a direct consequence of the above proposition and its corresponding proof, we deduce the following bounds for $\{\rho^{(k)}\}_k$, and $\{\pi^{(k)}\}_k$.

\begin{corollary}\label{coroestbi}  The next estimates hold
\begin{equation}\label{esrho1bi}
    \sup_{0\leq t\leq \tau}\|\nabla\rho^{(k)}(t)\|_{1,p}\leq \text{exp}(1)\|\nabla\rho_0\|_{1,p}=L_1,
\end{equation}

\begin{equation}\label{esp1bi}
    \sup_{0\leq t\leq \tau}\|\nabla \pi^{(k)}(t)\|_{2,p}\leq  K(L_1)A^2=L_2. 
\end{equation}
\end{corollary}

\subsection{Convergence results}

To show the convergence of the successive approximation constructed in the previous section, we consider the sequences of differences given by: $\sigma\kk=\rho\kk-\rho\ek$, $h\kk=u\kk-u\ek$, $q\kk=\pi\kk-\pi\ek $, $w\kk= v\kk- v\ek$. 
\begin{remark}
Subtracting the equations in each of the problems \eqref{iterhobi}, \eqref{itepbi} and \eqref{iteubi} corresponding to the steps $(k)$ and $(k-1)$, we obtain 
\begin{equation}\label{defsigmabi}
\begin{cases} 
\sigma^{(k)}_t + v^{(k-1)} \cdot \nabla \sigma^{(k)} =-\nabla \rho\ek \cdot w\ek, &\\ 
\sigma^{(k)}|_{t=0}=0,&
\end{cases}
\end{equation}
\begin{align}\label{defqbi}
\begin{cases}
\text{div}\left(\frac{1}{\rho^{(k)}}\nabla q^{(k)}\right)&=\text{div}\left( \frac{\sigma\kk}{\rho\kk \rho\ek}\nabla \pi\ek \right)-\sum_{i, j=1}^3   w^{(k-1),i}_{x_j} v^{(k-1),j}_{x_i}\\
&- \sum_{i, j=1}^3  v_\xxj^{(k-2),i}w_\xxi^{(k-1),j},  
\end{cases}
\end{align}
\begin{align}\label{defhbi}
\begin{cases} 
 h_t^{(k)} + v^{(k-1)}\cdot \nabla h^{(k)} + \frac{\nabla q^{(k)}}{\rho^{(k)}}=\frac{\sigma\kk}{\rho\ek \rho\kk}\nabla \pi\ek - w\ek\nabla u\ek \\
 -   w^{(k-1)}\nabla \W^{Q},  \\ 
h^{(k)}|_{t=0} =0.
\end{cases}
\end{align}
\end{remark}
In the following three technical lemmas we deduce some estimates of the new variables $\sigma\kk$, $q\kk$ and $h\kk$, bounded in terms of  $w\kk$. We omit the proofs  that follows the similar arguments of the propositions \ref{sigmaest}, \ref{qest} and \ref{hest}.
\begin{lemma}\label{sigmaestbi}
    Let $\sigma\kk$ be the solution of the system \eqref{defsigmabi}, then 
    \begin{align}
    \|\sigma\kk(t)\|_{1,p}\leq 
    L_5 \intt  \| w\ek(s)\|_{1,p}\,ds,
    \end{align}  
for all $t\in [0, \tau]$.
\end{lemma}
\begin{lemma}\label{qestbi}
Let $q\kk$ be the solution of the system \eqref{defqbi}, then     
\begin{equation*}
    \|\nabla q\kk(t)\|_{1,p}\leq L_6\left(   \|\sigma\kk(t)\|_{1,p}+\|w\ek(t)\|_{1,p}\right),
\end{equation*}
for all $t\in [0, \tau]$.
\end{lemma}
\begin{lemma}\label{hestbi}
Let $h\kk$ be the solution of the system \eqref{defhbi}, then 
\begin{equation*}
    \|h\kk(t) \|_{1,p}\leq L_9 \intt \left( \| \nabla q\kk(s)\|_{1,p}+  \|w\ek(s) \|_{1,p}+   \|w\ek(s) \|_{1,p}  \| \W^{Q}\|_{k,2} + \,\|\sigma\kk (s)\|_{1,p}\right)\, ds,
\end{equation*}
for all $t\in [0, \tau]$.
\end{lemma}
At this point, we can finally prove the convergence of the sequence $\{v^{(k)}\}_k$, which will be the key to proving the convergence of the other sequence of approximations. 
\begin{proposition}\label{vk-limitbi}
    The sequence $\{v^{(k)}\}_k$ converge in $C([0,\tau];W^{1,p}(\RR^3))$ to a random variable $v\in C([0, \tau]; W^{2,p}(\RR^3))$.
\end{proposition}
\begin{proof}
 We have deduce from \eqref{eqphibi}
\begin{align*}
    \| w\kk \|_{1,p}&\leq \| h\kk \|_{1,p}+ \| \nabla (\phi\kk-\phi\ek) \|_{1,p}\leq c_9 \| h\kk \|_{1,p},
\end{align*}
using the previous lemmas 
\begin{align*}
    \|w\kk(t)\|_{1,p}\leq L_{10}(\| \W^{Q}\|_{k,2}) \intt  \ \| w\ek(s) \|_{1,p} \,ds,
\end{align*}
from which it follows that
\begin{align*}
    \|w\kk(t)\|_{1,p}&\leq L_{10}^{k-1} \frac{t^{k-1}}{(k-1)!} \sup_{0\leq s\leq t}\| w^1(s) \|_{1,p},,
\end{align*}
then 
\begin{align*}
    \sup_{0\leq t\leq \tau}\|w\kk(t)\|_{1,p}
    &\leq A\,L_{10}^{k-1}  \frac{\tau^{k-1}}{(k-1)!} ,
\end{align*}
therefore
\begin{align}\label{wk-seriesbi}
\sum_{k=1}^\infty\|w\kk\|_{C([0,\tau];W^{1,p}(\mathbb{R}^3))}<\infty.
\end{align}  
This implies that the sequence $\{v^{(k)}\}_k$ is Cauchy in $C([0,\tau]; W^{1,p}(\mathbb{R}^3))$, and by Proposition~\ref{propbouvbi}, it is also bounded in $C([0,\tau]; W^{2,p}(\mathbb{R}^3))$.
Using the same argument as in Proposition~\ref{vk-limit}, we can deduce that $\{v^{(k)}\}_k$ converges in $C([0,\tau]; W^{1,p}(\mathbb{R}^3))$ to an element $v \in C([0,\tau]; W^{2,p}(\mathbb{R}^3))$
\end{proof}

Now, we present the convergence results for the remaining sequences of approximations and prove that the corresponding limits are the desired solutions to our problems. We omit the proofs, we continue with the same arguments as in the Lemmas and Propositions \ref{limit-rhok}, \ref{solution-rho},  \ref{limit-pik}, \ref{solution-u} and \ref{solution-phi}.
\begin{lemma}\label{limit-rhokbi}
    The sequence $\{\rho^{(k)}\}_k$ converges in $C([0,\tau]\times \mathbb{R}^3)$ to a random variable $\rho\in C([0, \tau]\times\mathbb{R}^3)$ such that $\nabla \rho \in C([0, \tau], W^{1,p}(\mathbb{R}^3))$. Moreover, 
$\rho$ is  a weak solution of the equation
    \begin{align}
    \left\{
    \begin{aligned}
       & \rho_t +  v\cdot \nabla \rho =0\\
       & \rho|_{t=0}=\rho_0(x)
       \end{aligned}
       \right.\,,
 \end{align}
 where   $v$ is the random variable from the Proposition \ref{vk-limitbi}.
\end{lemma}
We continue now dealing with the convergence of the sequences $\{\nabla \pi^{(k)}\}_k$ and $\{ u^{(k)}\}_k$.
\begin{proposition}
The sequence $\{\nabla \pi^{(k)}\}_k$ converges in $C([0,\tau]; W^{1,p}(\RR^3))$ to a random variable $\nabla \pi \in C([0,\tau]; W^{2,p}(\RR^3))$.  Moreover, $\nabla \pi$ is a weak solution to the problem
 \begin{align}\label{divbi}
     \text{div}( \rho^{-1}(t)   \nabla \pi(t))=-\text{div}(v(t)\cdot \nabla v(t)),
 \end{align}
 for every $t\in [0,\tau]$.
\end{proposition}
\begin{proposition}\label{solution-ubi}
 The sequence $\{ u^{(k)}\}_k$ converges in $C([0,\tau]; W^{1,p}(\RR^3))$ to a random variable $u \in C([0,\tau]; W^{2,p}(\RR^3))$. Moreover, $u$ is a weak solution to the problem
\begin{equation}\label{ubi}
\begin{cases}  
 u_t + (v\cdot \nabla) u  
 + (v\cdot \nabla) \W^{Q}+\frac{\nabla \pi}{\rho}=0 \\
u|_{t=0} =v_0(x). &
\end{cases}
\end{equation}   
\end{proposition}

The following Lemma allow us to define the vector field candidate to be the solution of the second equation of \eqref{eq0bi}.
\begin{lemma}\label{solution-phi-RA}
The sequence $\{ \nabla\phi^{(k)}\}_k$ converges in $C([0,\tau]; W^{1,p}(\RR^3))$ to a random variable $\nabla \phi \in C([0,\tau]; W^{1,p}(\RR^3))$. Besides that, if $v$ and $u$ are the random variables from Proposition \ref{vk-limitbi} and Proposition \ref{solution-ubi}, respectively, then $\nabla \phi$ satisfies
\begin{align}\label{D2-eq1-RA}
& \tilde{u}= u + \W^{Q} \quad \text{and} \quad v=\tilde{u}-\nabla \phi,
\end{align}
with $\phi$ being a weak solution of the elliptic equation
\begin{align*}
\Delta \phi = \text{div} \ \tilde{u}.
\end{align*}
\end{lemma}

\subsection{Existence and uniqueness }

In this section, we state the existence and uniqueness result for Problem \eqref{eq0bi}. We only outline the main ideas of the proofs, since the arguments are similar to those used in Propositions \ref{existence} and \ref{uniqueness}.

\begin{proposition}[Existence of solutions]\label{existencebi}
  If $\Tilde{u}=u+ \W^{Q} $, then  ($\rho$, $\nabla \pi$, $\tilde{u}$, $\tau$) determines a local pathwise solution for Problem \eqref{eq0bi} in the sense of Definition \ref{deflocalsolutionbi}.  
\end{proposition} 
\begin{proof}
The key step is to prove that $\tilde{u}=v$. In fact, by Proposition \ref{solution-ubi} the process $\Tilde{u}=u+ \W^{Q}$ satisfies
 \begin{equation}\label{u-tilde}
\begin{cases}
d\tilde{u} + (v\cdot \nabla) \tilde{u} \ dt+\frac{\nabla \pi}{\rho} \ dt=d\W^Q \\
\tilde{u}|_{t=0} =v_0(x). & 
\end{cases}
\end{equation}
By Eq. \eqref{D2-eq1-RA} and applying the divergence operator in the above equation, we have
\begin{align*}
    (\text{div} \ \tilde{u})_t+ v\cdot\nabla(\text{div} \ \tilde{u} )=-\sum_{i,j=1}^3 v^i_\xxj\phi_\xij.
\end{align*}
From here, we proceed by following the same arguments as in Proposition~\ref{existence}. Therefore, we conclude that $\tilde{u}=v$. Thus, by Lemma \ref{limit-rhokbi} and Eq. \eqref{u-tilde}, we obtain the desired result.
\end{proof}

\begin{proposition}\label{uniquenessbi}
    The solution to the problem \eqref{eq0bi} is pathwise unique.
\end{proposition}
\begin{proof}
We consider $(\rho^1, \nabla \pi^1, v^1, \tau_1)$ and $(\rho^2, \nabla \pi^2, v^2, \tau_2)$ two solutions of the problem \eqref{eq0bi} with the same initial condition $(\rho_0, v_0)$. Thus, we define the stopping time  $\tau= \tau_1\wedge  \tau_2$ and the random variables
\begin{equation*}
    v^1=u^1+ \W^{Q} \quad \text{and} \quad v^2=u^2+ \W^{Q}, \quad\quad t\in [0,\tau ]
\end{equation*}
Since $(\rho^1, \nabla \pi^1, v^1, \tau_1)$ and $(\rho^2, \nabla \pi^2, v^2, \tau_2)$ are solutions of \eqref{eq0bi}, then $(\rho^1, \nabla \pi^1, u^1)$ and $(\rho^2, \nabla \pi^2, u^2)$, for all $t\in [0,\tau ]$, respectively, verify
 \begin{equation*}
    \begin{cases}
        \rho^1_t +  v^1 \cdot \nabla \rho^1 =0,\\
         u^1_t + (v^1\cdot \nabla) u^1 + (v^1\cdot \nabla) \W^{Q}+\frac{\nabla \pi^1}{\rho^1}=0 ,\,\\
       \text{div } u^1=0,
    \end{cases}
\end{equation*}
and
\begin{equation*}
    \begin{cases}
        \rho^2_t +  v^2 \cdot \nabla \rho^2 =0,\\
         u^2_t + (v^2\cdot \nabla) u^2 + (v^2\cdot \nabla) \W^{Q}+\frac{\nabla \pi^2}{\rho^2}=0 ,\,\\
       \text{div } u^2=0,
    \end{cases}
\end{equation*}
Setting $\sigma=\rho^2-\rho^1$,  $q=\pi^2-\pi^1 $ and $w= u^2- u^1=v^2- v^1$, with $\sigma(0)=w(0)=0$, and subtracting the previous equations, we arrive at
\begin{equation}\label{diffsolutions-RA}
    \begin{cases}
   \sigma_t + v^2\cdot \nabla \sigma =-w\cdot \nabla \rho^1,\\   
   w_t + v^2\cdot \nabla w + \frac{\nabla q}{\rho^2} =\frac{\sigma}{\rho^2\rho^1}\nabla \pi^1- w\cdot \nabla u^1-(w\cdot \nabla) \W^{Q},\\
 \text{div }w=0.
    \end{cases}
\end{equation}
From now on, with the help of Lemmas \ref{sigmaestbi}, \ref{qestbi} and \ref{hestbi}, we continuous with the same arguments as in the proof of Proposition \ref{uniqueness} and hence we show that $(\rho^1, \nabla \pi^1, v^1)=(\rho^2, \nabla \pi^2, v^2)$ on $[0,\tau]$, as claimed.
\end{proof}

\subsection{Maximal pathwise solution}

In the same way as in the Section \ref{section-maximal}, it is possible to establish the existence of a unique maximal solution to Problem \eqref{eq0bi}. Since the arguments are analogous to those presented in Section \ref{section-maximal}, we state the result below without proof.
\begin{proposition}
    There exists a unique maximal pathwise solution to Problem~\eqref{eq0bi}, in the sense of Definitions~\ref{defmaximalsolutionbi} and~\ref{maxuniqueness}, with initial data $(\rho_0, v_0)$.
\end{proposition}

\medskip

Author Christian Olivera is partially supported by FAPESP by the grant  $2020/04426-6$,  by FAPESP-ANR by the grant Stochastic and Deterministic Analysis for Irregular Models$-2022/03379-0$ and  CNPq by the grant $422145/2023-8$.  Claudia Espitia is partially supported by FAPESP by the grant  2023/06709-3 and 2024/09988-3.


\begin{thebibliography}{99}
\scriptsize	
\bibitem{Bagnara-Maurelli-Xu}
M. Bagnara, M. Maurelli, F. Xu. {\it No blow-up by 
Itô noise for the Euler equations}. Electronic Journal of Probability 30, 1–29, 2025.

\bibitem{Veiga}
H. Beirão da Veiga, A. Valli. {\it On the Euler equations for nonhomogeneous fluids II}. Journal of Mathematical Analysis and Applications 73, 1980.

\bibitem{Veiga2}
H. Beirão da Veiga, A. Valli. {\it Existence of $C^{\infty}$ solutions of the Euler equations for nonhomogeneous fluids}. Communications in Partial Differential Equations 5, 1980.

\bibitem{Bessa}
H. Bessaih. {\it Martingale solutions for stochastic Euler equations}. Stochastic Analysis and Applications 17, 1999. 

\bibitem{Bessa2}
H. Bessaih, F. Flandoli. {\it $2$-D Euler equation perturbed by noise}. NoDEA Nonlinear Differential Equations and Applications 6, 1999.



\bibitem{Breit2}
D. Breit, E. Feireisl, M. Hofmanová. {\it On solvability and ill-posedness of the compressible Euler system subject to stochastic forces}. Analysis and PDE 13(2), 371–402, 2020.


\bibitem{Bre}
Z. Brzeźniak, B. Goldys, M. Ondreját. {\it Stochastic geometric partial differential equations}. New Trends in Stochastic Analysis and Related Topics, 1–32, Interdisciplinary Mathematical Sciences 12, World Scientific, Hackensack, NJ, 2012.

\bibitem{Bre4}
Z. Brzeźniak, B. Maslowski, J. Seidler. {\it Stochastic nonlinear beam equations}. Probability Theory and Related Fields 132, 119–149, 2005.

\bibitem{Bre3}
Z. Brzeźniak, F. Flandoli, M. Maurelli. {\it Existence and uniqueness for stochastic 2D Euler flows with bounded vorticity}. Archive for Rational Mechanics and Analysis 221, 2016. 

\bibitem{Bre2}
Z. Brzeźniak, S. Peszat. {\it Stochastic two-dimensional Euler equations}. Annals of Probability 29, 2001. 

\bibitem{Chi}
E. Chiodaroli, E. Feireisl, F. Flandoli. {\it Ill-posedness for the full Euler system driven by multiplicative white noise}. arXiv:1904.07977, 2019.

\bibitem{Correa}
J. Correa, C. Olivera. {\it From stochastic Hamiltonian systems to stochastic compressible Euler equations}. arXiv:2301.08101, 2023.

\bibitem{Correa2}
J. M. Correa, J. D. L. Acevedo, and C. Olivera, 
{\it From hamiltonian systems to compressible euler equation driven by additive hölder noise}, 
Infinite Dimensional Analysis, Quantum Probability and Related Topics, 28,  2025



\bibitem{Chae}
D. Chae, J. Lee. {\it Local existence and blow-up criterion of the inhomogeneous Euler equations}. Journal of Mathematical Fluid Mechanics 5, 2003.

\bibitem{Danchin}
R. Danchin. {\it On the well-posedness of the incompressible density-dependent Euler equations in the $L^{p}$ framework}. Journal of Differential Equations, 2010.


\bibitem{Fedrizzi}
E. Fedrizzi, F. Flandoli. {\it Noise prevents singularities in linear transport equations}. Journal of Functional Analysis 264(6), 1329–1354, 2013. 


\bibitem{Ferreira}
L.C.F. Ferreira, D.F. Machado. {\it On the well-posedness in Besov–Herz spaces for the inhomogeneous incompressible Euler equations}. Dynamics of Partial Differential Equations 21, 2024. 

\bibitem{Vicol}
N. Glatt-Holtz, V.C. Vicol. {\it Local and global existence of smooth solutions for the stochastic Euler equation with multiplicative noise}. Annals of Probability 42(1), 80–145, 2014. 

\bibitem{Flandoli2}
F. Flandoli, M. Gubinelli, E. Priola. {\it Well-posedness of the transport equation by stochastic perturbation}. Inventiones Mathematicae 180(1), 1–53, 2010.

\bibitem{Goodair-Crisan}
D. Goodair, D. Crisan. {\it Stochastic calculus in infinite dimensions and SPDEs}. Springer Nature, 2024.

\bibitem{Goodair-Crisan-Lang}
D. Goodair, D. Crisan, O. Lang. {\it Existence and uniqueness of maximal solutions to SPDEs with applications to viscous fluid equations}. Stochastics and Partial Differential Equations: Analysis and Computations, 1–64, 2023.



\bibitem{Has}
B. Haspot. {\it Well-posedness for density-dependent incompressible fluids with non-Lipschitz velocity}. Annales de l’Institut Fourier 62, 2012. 

\bibitem{ITO}
S. Itoh, A. Tani. {\it Solvability of nonstationary problems for nonhomogeneous incompressible fluids and the convergence with vanishing viscosity}. Tokyo Journal of Mathematics 22, 1999.


\bibitem{ITO2}
S. Itoh. {\it Cauchy problem for the Euler equations of a nonhomogeneous ideal incompressible fluid}. Journal of the Korean Mathematical Society 31, 1994.

\bibitem{ITO3}
S. Itoh. {\it Cauchy problem for the Euler equations of a nonhomogeneous ideal incompressible fluid II}. Journal of the Korean Mathematical Society 32, 1995.

\bibitem{KIM}
J.U. Kim. {\it On the stochastic Euler equations in a two-dimensional domain}. SIAM Journal on Mathematical Analysis 33, 2002. 

\bibitem{KIM2}
J.U. Kim. {\it Existence of a local smooth solution in probability to the stochastic Euler equations in $\mathbb{R}^{3}$}. Journal of Functional Analysis, 2009. 

\bibitem{LS}
O.A. Ladyzhenskaya, C.A. Solonnikov. {\it Unique solvability of an initial-and boundary-value problem for viscous incompressible nonhomogeneous fluids}. Journal of Soviet Mathematics 9, 1978.

\bibitem{L}
J.F. Le Gall. {\it Brownian motion, martingales, and stochastic calculus}. Springer International Publishing, 2016.


\bibitem{Mi}
R. Mikulevicius, G. Valiukevicius. {\it On stochastic Euler equation}. Lietuviu Matematikos Rinkinys 38, 1998.  




\end{thebibliography}
\end{document}